\documentclass[11pt,reqno]{amsart}
\usepackage[utf8]{inputenc}
\usepackage{amssymb}
\usepackage[all]{xy}
\usepackage{xcolor}
\usepackage{hyperref}
\newtheorem{thm}{Theorem}[section]
\newtheorem{lem}[thm]{Lemma}
\newtheorem{prop}[thm]{Proposition}
\newtheorem{cor}[thm]{Corollary}

\theoremstyle{definition}
\newtheorem{dfn}[thm]{Definition}

\theoremstyle{remark}
\newtheorem{remark}[thm]{Remark}

\newcommand{\CA}{{\mathcal{A}}}

\newcommand{\CF}{{\mathcal{F}}}
\newcommand{\CG}{{\mathcal{G}}}

\newcommand{\CU}{{\mathcal{U}}}

\newcommand{\CN}{{\mathcal{N}}}
\newcommand{\CI}{{\mathcal{I}}}

\newcommand{\CL}{{\mathcal{L}}}

\newcommand{\CZ}{{\mathcal{Z}}}
\newcommand{\CB}{{\mathcal{B}}}
\newcommand{\CR}{{\mathcal{R}}}

\newcommand{\CP}{{\mathcal{P}}}
\newcommand{\CW}{{\mathcal{W}}}

\newcommand{\af}{\alpha}
\newcommand{\bt}{\beta}
\newcommand{\gm}{\gamma}
\newcommand{\dt}{\delta}

\newcommand{\ld}{\lambda}

\newcommand{\sm}{\sigma}

\newcommand{\Z}{{\mathbb{Z}}}

\newcommand{\N}{{\mathbb{N}}}

\newcommand{\reg}{\operatorname{reg}}

\newcommand{\WCB}{{\widehat{\mathcal{B}}}}

\makeatletter
\@namedef{subjclassname@2020}{%
  \textup{2020} Mathematics Subject Classification}
\makeatother

\begin{document}


\title[$C^*$-algebras of GBDS as partial crossed products]
{$C^*$-algebras of generalized Boolean dynamical systems as partial crossed products}

\author[Gilles G. de Castro]{Gilles G. de Castro}
\address{Departamento de Matem\'atica, Universidade Federal de Santa Catarina, 88040-970 Florian\'opolis SC, Brazil.} \email{gilles.castro@ufsc.br}

\author[E. J. Kang]{Eun Ji Kang}
\address{
Research Institute of Mathematics, Seoul National University, Seoul 08826, 
Korea} \email{kkang3333\-@\-gmail.\-com}

\thanks{The second named author was supported by  Basic Science Research Program through  the National Research Foundation of Korea (NRF) funded by the Ministry of Education (No. NRF-2020R1A4A3079066)}



\begin{abstract} 
In this paper, we realize $C^*$-algebras of generalized Boolean dynamical systems as partial crossed products. Reciprocally, we give some sufficient conditions for a partial crossed product to be isomorphic to a C*-algebra of a generalized Boolean dynamical system. As an application, we show that gauge-invariant ideals of C*-algebras of generalized Boolean dynamical systems are themselves C*-algebras of generalized Boolean dynamical systems.
\end{abstract}

\maketitle
\section{Introduction}
Ever since the work of Cuntz-Krieger \cite{CK}, there has been an interest in studying C*-algebras associated with dynamics on totally disconnected locally compact Hausdorff spaces. Usually, there is a combinatorial object associated with the dynamics, such as a square matrix of 0-1 in the case of Cuntz-Krieger algebras. In \cite{BP1}, Bates and Pask initiated the study of C*-algebras associated with labeled graphs and showed how several of the generalization of Cuntz-Krieger algebras fitted in their framework. As a different approach to study these algebras, Carlsen, Ortega and Pardo introduced C*-algebras of Boolean dynamical systems in \cite{COP}. The main point is that totally disconnected locally compact Hausdorff spaces is completely characterized by the Boolean algebra of compact-open sets via Stone duality. However, due to some hypothesis needed in \cite{COP}, not all C*-algebras of labeled spaces were included in their work. This was achieved later by Carlsen and the second named author in \cite{CaK2} with generalized Boolean dynamical systems.

When studying C*-algebras, it is important and fruitful to realize $C^*$-algebras using different models as one can benefit from the established theory about these models. Two of them are groupoids C*-algebras \cite{Ren} and partial crossed products \cite{ExelBook}. In a previous work \cite{CasK1}, the authors have given several groupoid models for the C*-algebras of generalized Boolean dynamical systems.

The first goal of this paper is to  realize $C^*$-algebras of generalized Boolean dynamical systems as partial crossed products.
Given a generalized Boolean dynamical system $(\CB,\CL,\theta, \CI_\af)$,
the boundary path space $\partial E$  arising from a topological correspondence $E$ associated with the generalized Boolean dynamical system  was  introduced in \cite{CasK1}. 
To achieve our goal, we construct a semi-saturated orthogonal partial action  of the free group $\mathbb{F}$ generated by $\CL$ on $\partial E$. 
We then prove that the groupoid  obtained from  the partial action  is isomorphic to the Renault-Deaconu groupoid $\Gamma(\partial E, \sm_E)$, where $\sm_E$ is a shift map on $\partial E$. As a consequence of this isomorphism and the results of \cite{CasK1}, we have that the partial crossed product $C^*$-algebra obtained from the partial action is isomorphic to the $C^*$-algebra of the generalized Boolean dynamical system.

The second goal of the paper is to give sufficient conditions for a partial crossed product to be modeled using generalized Boolean dynamical systems similar to what is done for labeled graphs in \cite{BCGW}. For that, we need a partial action that is similar to the one obtained for generalized Boolean dynamical systems, namely a semi-saturated orthogonal partial action of the free group on a totally disconnected locally compact Hausdorff space with an extra condition. The C*-algebraic version of this model is then obtained by observing that the algebras of continuous functions on totally disconnected locally compact Hausdorff spaces are exactly the class of commutative C*-algebras generated by projections. As an application of this construction, we show that gauge-invariant ideals of C*-algebras of generalized Boolean dynamical systems are themselves C*-algebras of generalized Boolean dynamical systems.

This paper is organized as follows. In Section \ref{sec: prelim}, we provide the necessary definitions and results. 
In Section \ref{sec: partial action partial E}, we recall the definition of the boundary path space $\partial E$
arising from a generalized Boolean dynamical system and  present some of properties of  $\partial E$. We then  define a partial action on $\partial E$ (Proposition \ref{prop:partial action})
 and characterize the  associated partial crossed products (Proposition \ref{generator of the p.c.p}). 
In Section \ref{sec: groupoids}, we prove that the transformation groupoid of the partial action is isomorphic to the boundary path groupoid of the generalized Boolean dynamical system (Theorem \ref{iso groupoids}) and as a consequence, we show that the partial crossed product given in the previous section is isomorphic to $C^*$-algebra of the generalized Boolean dynamical system (Corollary \ref{cor: iso}). In Section \ref{section model}, we give sufficient conditions for a partial action to be modeled by a generalized Boolean dynamical system (Theorem \ref{partial action model}) and we prove that certain partial crossed products are isomorphic to the C*-algebras of Boolean dynamical systems (Corollary \ref{partial model C*-algebra}). Finally, in Section \ref{sec: ideals}, we apply the results of Section \ref{section model} to study gauge-invariant ideals of C*-algebras of generalized Boolean dynamical systems (Corollary \ref{cor: gauge-invariant ideal}).

\section{Preliminaries}\label{sec: prelim}
\subsection{Filters and characters}\label{Filters and characters}

 A {\em filter} in a partially ordered set $P$ with least element 0 is a non-empty subset  $\xi$ of $P$ such that 
\begin{enumerate}
\item[(i)] $0 \notin \xi$,
\item[(ii)] if $ x \in \xi$ and  $x \leq y$, then $y \in \xi$,
\item[(iii)] if $x,y \in \xi$, there exists $z \in \xi$ such that $z \leq x$ and $z \leq y$.
\end{enumerate}
If $P$ is a (meet) semilattice, condition (iii) may be replaced by $x \wedge y$ if $x,y \in \xi$. 
An {\it ultrafilter} is a filter that is not properly contained in any filter. 

For a given  $x \in P$, we define
$$\uparrow x=\{y \in P: x \leq y\} ~\text{and}~ \downarrow x =\{y \in P: y \leq x\},$$
and for subsets $X, Y $ of $ P$, we define
$\uparrow X=\{y \in P: x \leq y ~\text{for some}~x \in X\}$
and $\uparrow_Y X= Y \cap \uparrow X$. The sets $\uparrow_Y x$, $\downarrow_Y x$, $\downarrow X$ and $\downarrow_Y X$ should have their obvious meaning.

We mean by $E$ a semilattice with 0. 
A {\it character} on $E$ is a nonzero function  $\phi$ from $ E$ to the Boolean algebra $ \{0,1\}$ such that $$ \phi(0)=0  ~\text{and}~  \phi(ef)=\phi(e)\phi(f)$$ for all $e,f \in E$. We denote by $\hat E_0$ the set of all characters on $E$. We view $\hat E_0$ as a topological space equipped with the product topology inherited from $\{0,1\}^{E}$. 
     It is easy to see that $\hat E_0$ is a locally compact totally disconnected Hausdorff space. 
     
      Given a filter $\eta$ in $E$, the map $$\phi_{\eta}: E \rightarrow \{0,1\} ~\text{ given by}~ \phi_{\eta}(e) = \left\{ 
\begin{array}{ll}
    1 & \hbox{if\ } e \in \eta, \\
    0 &  \hbox{otherwise} \\
    \end{array}
\right.$$ is a character. Conversely, for a character $\phi$ on $E$, the set $$\eta_{\phi}=\{e \in E: \phi(e)=1\}$$ is a filter. These correspondences are mutually inverses.

A character $\phi$ of $E$ is called an  {\it ultra-character} if the corresponding filter $\xi_{\phi}$ is an ultrafilter. 
We denote by $\hat E_{\infty}$ the set of all ultra-characters.

  Given $x \in E$, a set $Z  \subseteq \downarrow x$ is a {\em cover} for $x$ if for all non-zero $y \in \downarrow x$, there exists $z \in Z$ such that $z \wedge y \neq 0$.
    A character $\phi$ of $E$ is {\it tight} if for every $x \in E$ and every finite cover $Z$ for $x$, we have 
$$ \bigvee_{z \in Z} \phi(z) = \phi(x).$$
The set of all tight characters is denoted by $\hat E_{tight}$, and 
 called the {\it tight spectrum} of $E$. 
  It is a closed subspace of $\hat E_0$ containing $\hat E_{\infty}$ as a dense subspace \cite[Sect. 12]{Exel1}.

\subsection{Boolean algebras}

A {\em Boolean algebra} is a set $\CB$ with a distinguished element $\emptyset$ and maps $\cap: \CB \times \CB \rightarrow \CB$, $\cup: \CB \times \CB \rightarrow \CB$ and $\setminus: \CB \times \CB \rightarrow \CB$ such that $(\CB,\cap,\cup)$ is a distributive lattice, $A\cap\emptyset=\emptyset$ for all $A\in\CB$, and $(A\cap B)\cup (A\setminus B)=A$ and $(A\cap B)\cap (A\setminus B)=\emptyset$ for all $A,B\in\CB$. The Boolean algebra $\CB$ is called {\em unital} if there exists $1 \in \CB$ such that $1 \cup A = 1$ and $1 \cap A=A$ for all $A \in \CB$ (often, Boolean algebras are assumed to be unital and what we here call a Boolean algebra is often called a \emph{generalized Boolean algebra}).

We call $A\cup B$ the \emph{union} of $A$ and $B$, $A\cap B$ the \emph{intersection} of $A$ and $B$, and $A\setminus B$ the \emph{relative complement} of $B$ with respect to $A$. A subset $\CB' \subseteq \CB$ is called a {\em Boolean subalgebra} if $\emptyset\in\CB'$ and $\CB'$ is closed under taking unions, intersections and relative complements. A Boolean subalgebra of a Boolean algebra is itself a Boolean algebra.

We define a partial order on $\CB$ as follows: for $A,B \in \CB$,
\[
A \subseteq B ~~~\text{if and only if}~~~A \cap B =A.
\]
Then $(\CB, \subseteq)$ is a partially ordered set, and $A\cup B$ and $A\cap B$ are the least upper-bound and the greatest lower-bound of $A$ and $B$ with respect to the partial order $\subseteq$. If $A\subseteq B$, then we say that $A$ is a \emph{subset of} $B$.

A non-empty subset $\CI$ of $\CB$ is called  an {\em ideal} \cite[Definition 2.4]{COP} if 
\begin{enumerate}
\item[(i)] if $A, B \in \CI$, then $A \cup B \in \CI$,
\item[(ii)] if $A \in \CI$ and $ B \in \CB$, then   $A \cap B \in \CI$. 
\end{enumerate}

An ideal $\CI$ of a Boolean algebra $\CB$ is a Boolean subalgebra. For $A \in \CB$, the ideal generated by $A$ is defined by $\CI_A:=\{ B \in \CB : B \subseteq A\}.$

A filter in $\CB$ is {\em prime} if for every $B,B' \in \CB$ with $B \cup B'\in\xi$, we have that either $B \in \xi$ or $B' \in \xi$. We note that $\xi$ is an ultrafilter if and only if it is a prime filter, and we use this equivalence throughout the paper without further mention.

 Given a Boolean algebra $\CB$, we write $\widehat{\CB}$  for the set of all ultrafilters of $\CB$. Notice that if $A\in\CB\setminus\{\emptyset\}$, then $\{B\in\CB:A\subseteq B\}$ is a filter, and it then follows from Zorn's Lemma that there is an ultrafilter $\eta\in\widehat{\CB}$ that contains $A$. For $A\in\CB$, we let $$Z(A):=\{\xi\in\widehat{\CB}:A\in\xi\}$$ and we equip $\widehat{\CB}$ with the topology generated by $\{Z(A): A\in\CB\}$. Then $\widehat{\CB}$ is a totally disconnected locally compact Hausdorff space, $\{Z(A): A\in\CB\}$ is a basis for the topology, and each $Z(A)$ is compact and open.

\subsection{Boolean dynamical systems}

A map $\phi: \CB \rightarrow \CB'$ between two Boolean algebras is called a {\em Boolean homomorphism} (\cite[Definition 2.1]{COP}) if $\phi(A \cap B)=\phi(A) \cap \phi(B)$, $\phi(A \cup B)=\phi(A) \cup \phi(B)$, and $\phi(A \setminus B)=\phi(A) \setminus \phi(B)$ for all $A,B \in \CB$.

A map $\theta: \CB \rightarrow \CB $ is called an {\it action} (\cite[Definition 3.1]{COP})  on a Boolean algebra $\CB$ if it is a Boolean homomorphism with $\theta(\emptyset)=\emptyset$.

Given a set $\CL$ and any $n \in \N$, we define $\CL^n:=\{(\af_1, \dots, \af_n): \af_i \in \CL\}$,  $\CL^{\geq 1}=\cup_{n \geq 1} \CL^n$  and $\CL^*:=\cup_{n \geq 0} \CL^n$, where $\CL^0:=\{\emptyset \}$. For $\alpha\in\CL^n$, we write  $|\af|:=n$. For $\af=(\af_1, \dots, \af_n), \beta=(\beta_1,\dots,\beta_m) \in \CL^*$, we will usually write $\af_1 \dots \af_n$ instead of $(\af_1, \dots, \af_n)$ and use $\alpha\beta$ to denote the word $\af_1 \cdots \af_n\beta_1\dots\beta_m$ (if $\alpha=\emptyset$, then $\alpha\beta:=\beta$; and if $\beta=\emptyset$, then $\alpha\beta:=\alpha$). 
  For $1\leq i\leq j\leq |\af|$, we also denote by $\af_{i,j}$ the sub-word $\af_i\cdots \af_j$ of  $\af=\af_1\af_2\cdots\af_{|\af|}$, where $\af_{i,i}=\af_i$. If $j < i$, set $\af_{i,j} =\emptyset$.
  
 We also  let $\CL^\infty$ denote the set of infinite sequences with entries in $\CL$. 
If $x=(x_1,x_2,\dots)\in\CL^\infty$ and $n\in\N$, then we let $x_{1,n}$ denote the word $x_1x_2\cdots x_n\in\CL^n$.

A {\em Boolean dynamical system} is a triple $(\CB,\CL,\theta)$ where $\CB$ is a Boolean algebra, $\CL$ is a set, and $\{\theta_\af\}_{\af \in \CL}$ is a set of actions on $\CB$. For $\af=\af_1 \cdots \af_n \in \CL^{\geq 1}$, the action $\theta_\af: \CB \rightarrow \CB$ is defined as $\theta_\af:=\theta_{\af_n} \circ \cdots \circ \theta_{\af_1}$.  We  also define $\theta_\emptyset:=\text{Id}$.

For $B \in \CB$, we define
\[
\Delta_B^{(\CB,\CL,\theta)}:=\{\af \in \CL:\theta_\af(B) \neq
\emptyset \} ~\text{and}~  \ld_B^{(\CB,\CL,\theta)}:=|\Delta_B^{(\CB,\CL,\theta)}|.
\]
We will often just write $\Delta_B$ and $\ld_B$ instead of
$\Delta_B^{(\CB,\CL,\theta)}$ and $\ld_B^{(\CB,\CL,\theta)}$.

We say that $A \in \CB$ is {\em regular} (\cite[Definition 3.5]{COP})
if for any $\emptyset \neq B \in \CI_A$, we have $0 < \ld_B < \infty$.
If $A \in \CB$ is not regular, then it is called a {\em singular} set.
We write $\CB^{(\CB,\CL,\theta)}_{reg}$ or just $\CB_{reg}$ for the
set of all regular sets. Notice that $\emptyset\in\CB_{reg}$.

\subsection{Generalized Boolean dynamical systems and their $C^*$-algebras}\label{GBDS}

Let $(\CB,\CL,\theta)$ be a Boolean dynamical system and let 
\[
\mathcal{R}_\alpha^{(\CB,\CL,\theta)}:=\{A\in\mathcal{B}:A\subseteq\theta_\alpha(B)\text{ for some }B\in\mathcal{B}\}
\]
for each $\alpha \in \mathcal{L}$. Note that each $\CR_\af^{(\CB,\CL,\theta)}$ is an ideal of $\CB$. 
We will often, when it is clear which Boolean dynamical system we are working with, just write $\CR_\af$ instead of $\CR_\af^{(\CB,\CL,\theta)}$.
\

\begin{dfn}(\cite[Definition 3.2]{CaK2})\label{def:GBDS} 
A {\em generalized Boolean dynamical system} is a quadruple  $(\CB,\CL,\theta,\CI_\alpha)$ where  $(\CB,\CL,\theta)$ is  a Boolean dynamical system  and  $\{\CI_\alpha:\alpha\in\CL\}$ is a family of ideals in $\CB$ such that $\CR_\alpha\subseteq\CI_\alpha$ for each $\alpha\in\CL$.
\end{dfn}

\begin{dfn}\label{def:representation of RGBDS} 
Let $(\CB,\CL,\theta, \CI_\af)$ be a  generalized Boolean dynamical system. A {\it  $(\CB, \CL, \theta, \CI_\af)$-representation (or  Cuntz--Krieger representation of $(\CB, \CL,
\theta,\CI_\af)$)} is a family of projections $\{P_A:A\in\mathcal{B}\}$ and a family of partial isometries $\{S_{\alpha,B}:\alpha\in\mathcal{L},\ B\in\mathcal{I}_\alpha\}$ in a $C^*$-algebra $\CA$ such that for $A,A'\in\mathcal{B}$, $\alpha,\alpha'\in\mathcal{L}$, $B\in\mathcal{I}_\alpha$ and $B'\in\mathcal{I}_{\alpha'}$,
\begin{enumerate}
\item[(i)] $P_\emptyset=0$, $P_{A\cap A'}=P_AP_{A'}$, and $P_{A\cup A'}=P_A+P_{A'}-P_{A\cap A'}$;
\item[(ii)] $P_AS_{\alpha,B}=S_{\alpha,  B}P_{\theta_\af(A)}$;
\item[(iii)] $S_{\alpha,B}^*S_{\alpha',B'}=\delta_{\alpha,\alpha'}P_{B\cap B'}$;
\item[(iv)] $P_A=\sum_{\af \in\Delta_A}S_{\af,\theta_\af(A)}S_{\af,\theta_\af(A)}^*$ for all  $A\in \mathcal{B}_{reg}$. 
\end{enumerate}
\end{dfn}

Given a $(\CB, \CL, \theta, \CI_\af)$-representation $\{P_A, S_{\af,B}\}$ in a $C^*$-algebra $\CA$, we denote by $C^*(P_A, S_{\af,B})$ the $C^*$-subalgebra of $\CA$ generated by $\{ P_A,  S_{\af,B}\}$.
It is shown in \cite{CaK2} that there exists  a universal $(\CB, \CL, \theta, \CI_\af)$-representation $\{p_A, s_{\af,B}: A\in \CB, \af \in \CL ~\text{and}~ B \in \CI_\af\}$   in a  $C^*$-algebra. 
 We write $C^*(\mathcal{B},\mathcal{L},\theta, \CI_\af)$ for $C^*(p_A,s_{\af,B})$ and    call it the {\it  $C^*$-algebra of $(\CB,\CL,\theta,\CI_\alpha)$}. 
 When $(\CB,\CL,\theta)$ is a Boolean dynamical system, then we write
$C^*(\CB,\CL,\theta)$ for $C^*(\CB,\CL,\theta, \CR_\af)$ and call it
the \emph{$C^*$-algebra of $(\CB,\CL,\theta)$}.

For $\af=\af_1\af_2 \cdots \af_n \in \CL^{\geq 1}$, we define
\begin{align*}
\CI_\af:=\{A \in \CB : A \subseteq \theta_{\af_2 \cdots \af_n}(B)~\text{for some }~ B \in \CI_{\af_1}\}.
\end{align*}
 For $\af =\emptyset$, we  define $\CI_\emptyset := \CB$.

\begin{dfn}(\cite[Definition 3.6]{CaK2}) Let $\{P_A,\ S_{\alpha,B}: A\in\CB,\ \alpha\in\CL,\ B\in\CI_\alpha\}$ be a $(\CB,\CL,\theta,\CI_\alpha)$-representation.
For $\af=\af_1\af_2 \cdots \af_n \in \CL^{\geq 1}$ and $A \in \CI_{\af}$, we define
\[
S_{\af,A}:=S_{\af_1,B}S_{\af_2, \theta_{\af_2}(B)}S_{\af_3, \theta_{\af_2\af_3}(B)} \cdots S_{\af_n,A},
\] 
where $B \in \CI_{\af_1}$ is such that $A \subseteq \theta_{\af_2 \cdots \af_n}(B)$. For $\af = \emptyset$, we also define $S_{\emptyset, A}:=P_A$.
\end{dfn}

\begin{remark}\label{span}
Let $\{P_A,\ S_{\alpha,B}: A\in\CB, \alpha\in\CL ~\text{and}~ B\in\CI_\alpha\}$
be a $(\CB,\CL,\theta,\CI_\alpha)$-representation. 
\begin{enumerate}
\item  For $\af, \bt \in \CL^*$, $A \in \CI_\af$ and $B \in \CI_\bt$, we have the equality 
\[
S_{\af,A}^*S_{\bt,B}= \left\{ 
\begin{array}{ll}
    P_{A \cap B} & \hbox{if\ }\af =\bt, \\
    S_{\af', A \cap \theta_{\af'}(B)}^* & \hbox{if\ }\af =\bt\af', \\
    S_{\bt',B \cap \theta_{\bt'}(A)}   & \hbox{if\ } \bt=\af\bt', \\
    0 & \hbox{otherwise.} \\
\end{array}
\right.
\]
 We then have
that
\begin{align*}
C^*(P_A,S_{\alpha,B})&=\overline{\operatorname{span}}\{
S_{\af,A}S_{\bt,B}^*: ~\af,\bt \in \CL^* ~\text{and}~ A \in \CI_\af, B
\in \CI_\bt\}\\
&=\overline{{\rm \operatorname{span}}}\{S_{\af,A}S_{\bt,A}^*: \af,\bt
\in \CL^* ~\text{and}~ A \in \CI_\af\cap \CI_\bt \}.
\end{align*}
\item It follows from the universal property of $C^*(\CB,\CL,\theta,
\CI_\af)=C^*(p_A, s_{\af,B})$ that there is a strongly continuous
action $\gm:\mathbb T\to {\rm Aut}(C^*(\CB,\CL,\theta, \CI_\af))$,
which we call the {\it gauge action}, such that

\[
\gm_z(p_A)=p_A   \ \text{ and } \ \gm_z(s_{\af,B})=zs_{\af,B}
\]
for $A\in \CB$, $\af \in \CL$ and $B \in \CI_\af$.
 \end{enumerate}
\end{remark}

\subsection{An inverse semigroup}\label{An inverse semigroup}
 Let $(\CB, \CL,\theta, \CI_\af)$ be a generalized Boolean dynamical system and let
$$S_{(\CB, \CL,\theta, \CI_\af)}:=\{(\af, A, \bt): \af,\bt \in \CL^* ~\text{and}~ \emptyset \neq   A \in \CI_\af \cap \CI_\bt \} \cup \{0\}. $$
To simplify the notation, we write $S=S_{(\CB, \CL,\theta, \CI_\af)}$ when it is clear which generalized Boolean dynamical system we are working with.

A binary operation on $S$ is defined as follows: $s \cdot 0 = 0 \cdot s=0$ for all $s \in S$ and 
for $(\af, A, \bt)$ and $(\gm, B, \dt)$ in $S$, 

$$(\af, A, \bt) \cdot (\gm, B, \dt)=\left\{
 \begin{array}{ll}    (\af, A \cap B, \dt) & \hbox{if\ }\bt =\gm  ~\text{and}~ A\cap B \neq \emptyset, \\
 (\af\gm', \theta_{\gm'}(A)\cap B, \dt) & \hbox{if\ } \gm=\bt\gm' ~\text{and}~ \theta_{\gm'}(A)\cap B \neq \emptyset,\\
 (\af, A\cap \theta_{\bt'}(B) ,\dt\bt') & \hbox{if\ } \bt=\gm\bt'  ~\text{and}~ A\cap \theta_{\bt'}(B) \neq \emptyset,\\
          0 & \hbox{otherwise.}
                      \end{array}
                    \right.
$$
 If for a given $s=(\af, A, \bt) \in S$ we define $s^*=(\bt, A, \af)$, then the set $S$, endowed with the operation above, 
 is an  inverse semigroup with zero element 0 (\cite[Sect. 2.3]{BCM1}), whose semilattice of  idempotents is 
$$E(S):=\{(\af, A, \af): \af \in \CL^* ~\text{and}~ \emptyset \neq A \in \CI_\af\}\cup \{0\}.$$ 

The natural order in the semilattice $E(S)$ is given as follows: for $\af,\bt \in \CL^*$, $A \in \CI_\af$ and $B \in \CI_\bt$, we have 
 $$(\af, A, \af)\leq (\bt, B, \bt) ~\text{if and only if}~\af=\bt\af' ~\text{and}~A \subseteq\theta_{\af'}(B)$$ (\cite[Lemma 3.1]{CasK1}).

\subsection{Filters in $E(S)$}
From now on, we define  $\CW^{\geq 1}=\{\alpha\in\CL^{\geq 1}:\CI_\alpha\neq \{\emptyset\}\}$,  $\CW^*=\{\alpha\in\CL^*:\CI_\alpha\neq \{\emptyset\}\}$, $\CW^{\infty}=\{\alpha\in\CL^{\infty}:\alpha_{1,n}\in\CW^{\geq 1} ~\text{for all}~ n \geq 1\}$ and $\CW^{\leq\infty}=\CW^*\cup \CW^{\infty}$. 

Let $\af \in \CW^{\leq \infty}$ and $\{\CF_n\}_{0 \leq n \leq |\af|}$ be a family such that $\CF_n$ is a filter in $\CI_{\af_{1,n}}$ for every $n > 0$ and $\CF_0$ is either a filter in $\CB$ or $\CF_0=\emptyset$. The family $\{\CF_n\}_{0 \leq n \leq |\af|}$ is said to be {\it complete for $\af$} if 
$$\CF_n=\{A \in \CB :  \theta_{\af_{n+1}}(A) \in \CF_{n+1}\}$$
for all $0 \leq n <|\af|$.

\begin{thm}\label{filter-bij-complete pair}(\cite[Theorem 3.12]{CasK1}) Let $(\CB, \CL,\theta, \CI_\af)$ be a generalized Boolean dynamical system and $S$ be its associated inverse semigroup. Then there is a bijective correspondence between filters in $E(S)$ and pairs $(\af, \{ \CF_n\}_{0 \leq n \leq |\af|})$, where $\af \in \CW^{\leq \infty}$ and $\{\CF_n\}_{0 \leq n \leq |\af|}$ is a complete family for $\af$.
\end{thm}

Filters are of {\it finite type} if they are associated with pairs $(\af, \{\CF_n\}_{0 \leq n \leq |\af|})$ for which $|\af|< \infty$, and of {\it infinite type} otherwise.

In view of this, a filter $\xi$ in $E(S)$ with associated  $\af \in \CL^{\leq \infty}$ is denoted by $\xi^\af$ to stress $\af$; in addition, the filters in the complete family associated with $\xi^\af$ will be denoted by $\xi_n^\af$ (or simply $\xi_n$). Specifically, 
$$\xi^\af_n=\{A \in\CB: (\af_{1,n}, A, \af_{1,n}) \in\xi^\af\}.$$

 We denote by $\mathsf{T}$  the set of tight filters on $E:=E(S)$ and we equip  $\mathsf{T}$  with the topology induced from the topology of pointwise convergence of character, via the bijection between tight characters and tight filters given in subsection \hyperref[Filters and characters]{1.1}. 
Note then that $\mathsf{T}$ is (homeomorphic to) the tight spectrum $\hat{E}_{tight}$ of $E$. 

\begin{remark}\label{basis of T}
Using the bijection between filters and characters as well as the topology on characters given by pointwise convergence, we see that a basis of compact-open sets for the induced topology on $\mathsf{T}$ is given by sets of the form
\[V_{e:e_1\ldots,e_n}=\{\xi\in\mathsf{T}:e\in\xi,e_1\notin\xi,\ldots,e_n\notin\xi\},\]
where $e \in E$ and $\{e_1, \cdots, e_n\}$ is a finite (possibly empty) subset of $E$. See \cite[Sect. 2.2]{Lawson2012} for more details.
\end{remark}

\begin{thm}\label{char:tight}(\cite[Theorem 3.27]{CasK1})  Let $(\CB, \CL,\theta, \CI_\af)$ be a generalized Boolean dynamical system and $S$ be its associated inverse semigroup. Then the tight filters in $E(S)$ are :
\begin{enumerate}
\item[(i)] The ultrafilters of infinite type.
\item[(ii)] The filters of finite type $\xi^\af$ such that $\xi_{|\af|}$ is an ultrafilter and   $A\notin\CB_{reg}$ for all $A \in \xi_{|\af|}$.
\end{enumerate}
\end{thm}

For each  $\af \in  \CW^*$, 
 we write  $X_\af$ instead of $\widehat{\CI_\af}$ for the set of all ultrafilters in $\CI_\af$ to match our notations with \cite{CasK1}. Note that $X_\emptyset$ denotes the set of ultrafilters in $\CI_\emptyset=\CB$.  
  For $A \in \CI_\af$, we let $$Z(\af, A):=\{\CF \in X_\af: A \in \CF\}$$ and equip $X_\af$
 with the topology generated by $\{Z(\af, A): A\in\CI_\af\}$.
We also  consider the set $X_\emptyset\cup\{\emptyset\}$ with a suitable topology.  If $\CB$ is unital, the topology is such that $\{\emptyset\}$ is an isolated point. If $\CB$ is not unital, then $\emptyset$ plays the role of the point at infinity in the one-point compactification of $X_\emptyset$. 

Given $\af, \bt \in \CW^{\geq 1}$, since the  action $$\theta_\bt : \CI_\af \to \CI_{\af\bt}$$ is a proper Boolean  homomorphism (\cite[Lemma 3.21]{CasK1}), there is its dual morphism 
$$f_{\af[\bt]}: X_{\af\bt} \to X_\af$$ given by 
 $f_{\af[\bt]}(\CF)=\{A \in \CI_\af: \theta_\bt(A) \in \CF\}$.
When $\af=\emptyset$, if $\CF \in X_\bt$, then $\{A \in \CB: \theta_\bt(A) \in \CF\}$ is either an ultrafilter in $\CI_\emptyset(=\CB)$ or  the empty set. We can therefore consider 
$f_{\emptyset[\bt]}: X_\bt \to X_\emptyset \cup \{\emptyset\}.$ Notice that for $\af \in \CW^*$ and $ \bt \in \CW^{\geq 1}$,  $f_{\af[\bt]}$ is  continuous  (\cite[Lemma 3.23]{CasK1}), and that  $f_{\af[\bt\gm]}=f_{\af[\bt]}\circ f_{\af\bt[\gm]}$ for all $\af\in\CW^{*}$ and $\beta,\gamma\in\CW^{\geq 1}$ such that $\alpha\beta\gamma\in\CW^{\geq 1}$.

For $\af \in \CL^{\geq 1}$ and $\bt \in \CL^*$ such that $\af\bt \in \CW^*$, consider an open  subspace
$$X_{(\af)\bt}:=\{\CF \in X_\bt:\CF \cap \CI_{\af\bt}\neq \emptyset  \}$$ of $X_\bt$. Then there is a continuous map 
 $g_{(\af)\bt}: X_{(\af)\bt} \to X_{\af\bt}$ defined  by 
$$g_{(\af)\bt}(\CF):=     \CF \cap \CI_{\af\bt}  $$
for each  $\CF \in X_{(\af)\bt}$ (\cite[Lemm 4.6]{CasK1}). For $\af =\emptyset$, define $X_{(\emptyset)\bt}=X_{\bt}$ and let $g_{(\emptyset)\bt}$ denote the identity function on $X_\bt$.

Also, for $\af \in \CL^{\geq 1}$ and $\bt \in \CL^*$ such that $\af\bt \in \CW^*$, 
there is a continuous map 
$$h_{[\af]\bt}:X_{\af\bt} \to X_{(\af)\bt}$$ defined by 
$$h_{[\af]\bt}(\CF)=\uparrow_{\CI_\bt}\CF$$ for each ultrafilter $\CF \in X_{\af\bt}$.
We note that   $h_{[\af]\bt}: X_{\af\bt} \to X_{(\af)\bt}$ and $g_{(\af)\bt}:X_{(\af)\bt} \to X_{\af\bt}$ are mutually inverses (\cite[Lemma 4.8]{CasK1}).

\subsection{Partial actions}\label{section partial actions}

\begin{dfn} 
A {\it partial action} of a group $G$ on a topological space $X$ is a pair $\Phi=(\{U_t\}_{t \in G}, \{\phi_t\}_{t \in G})$ consisting of a collection $\{U_t\}_{t \in G}$ of open subsets of $X$  and a collection $\{\phi_t\}_{t \in G}$ of homeomorphisms, $$\phi_t: U_{t^{-1}} \to U_t,$$
such that 
\begin{enumerate}
\item $U_e=U_{e^{-1}}=X$ and $\phi_e$ is the identity on $X$,
\item $\phi_s(U_{s^{-1}} \cap U_t)=U_s \cap U_{st}$,
\item $\phi_s(\phi_t(x))=\phi_{st}(x)$ for every $x \in U_{t^{-1}} \cap U_{(st)^{-1}}$.
\end{enumerate}
\end{dfn}

If the partial action is given by the free group $\mathbb{F}$ on a set of generators, then the partial action is {\it semi-saturated} if 
$$\phi_s \circ \phi_t=\phi_{st}$$
for every $s,t \in \mathbb{F}$ such that $|st|=|s|+|t|$, and {\it orthogonal} if $U_a \cap U_b = \emptyset$ for $a, b$ in the set of generator with $a \neq b$.

We refer the reader to \cite{ExelBook} for more details on partial actions and the construction of the partial crossed product.

We describe the groupoid of a partial action as in \cite{A}. Let $\Phi=(\{U_t\}_{t \in G}, \{\phi_t\}_{t \in G})$ be a partial action of  $G$ on $X$. 
Then, 
$$G \ltimes_\varphi X :=\{(x,t,y) \in X \times G \times X: y \in U_{t^{-1}} ~\text{and}~ x=\varphi_t(y)\}$$ is a groupoid with products and inverses given by 
$$(x,t,y)(y,s,z)=(x,ts,z) ~\text{and}~ (x,t,y)^{-1}=(y, t^{-1}, x).$$
We give $G \ltimes_\varphi X$ the topology inherited from the product topology on $X \times G \times X$.

\section{A partial action on $\partial E$ and the associated partial crossed product}\label{sec: partial action partial E}
In this section, we define a partial action from a generalized Boolean dynamical system $(\CB, \CL,\theta, \CI_\af)$. The group that acts is the free group generated by $\CL$ and the space where it acts is the boundary path space of a topological correspondence defined previously in \cite{CasK1}. 
We then prove that the partial crossed product associated with this partial action is generated by an appropriate set of characteristic functions and the generators corresponding to the elements of the free group.

\subsection{The boundary path space $\partial E$}

Let $(\CB, \CL,\theta, \CI_\af)$ be a generalized Boolean dynamical system.  Recall that  $X_\emptyset=\WCB$ is equipped with  the topology generated by $\{Z(A): A\in\CB\}$, where 
$Z(A)=\{\xi\in\widehat{\CB}:A\in\xi\}$, and that 
$X_\af=\widehat{\CI_\af}$ is equipped with the topology generated by  $\{Z(\af, A): A\in\CI_\af\}$, where $Z(\af, A)=\{\xi \in \widehat{\CI_{\af}}: A \in \xi\}$.
 
We let $$E^0_{(\CB,\CL,\theta,\CI_\alpha)}:=X_\emptyset ~\text{and}~	F^0_{(\CB,\CL,\theta,\CI_\alpha)}:=X_\emptyset\cup\{\emptyset\}$$
as  topological spaces. 
We also let
$$	E^1_{(\CB,\CL,\theta,\CI_\alpha)}:=\bigl\{e^\alpha_\eta:\alpha\in\CL,\ \eta\in X_\alpha\bigr\}
$$
and equip $E^1_{(\CB,\CL,\theta,\CI_\alpha)}$ with the topology generated by 
$$ \mathbb{V}:=\bigcup_{\alpha\in\CL} \{Z^1(\af, B):B\in\CI_\alpha\}, $$
where $Z^1(\af, B):=\{e^\af_\eta: \eta \in X_\af  , B \in \eta\}.$  Note that $E^1_{(\CB,\CL,\theta,\CI_\alpha)}$ is homeomorphic to the disjoint union of the family $\{X_{\af}\}_{\af\in\CL}$.

\begin{prop}(\cite[Proposition 7.1]{CasK1})\label{def:topological correspondence} Let $(\CB,\CL,\theta, \CI_\af)$ be a generalized Boolean dynamical system and let $E^0:=E^0_{(\CB,\CL,\theta,\CI_\alpha)}$, $F^0:=F^0_{(\CB,\CL,\theta,\CI_\alpha)}$ and $E^1:=E^1_{(\CB,\CL,\theta,\CI_\alpha)}$ be as above. If we define the maps
 $d:E^1\to E^0$ and $r:E^1\to F^0$ by 
\[
	d(e^\alpha_\eta)=h_{[\alpha]\emptyset}(\eta) \text{ and } r(e^\alpha_\eta)=f_{\emptyset[\af]}(\eta),
\]
 then $(E^1, d,r)$ is a topological correspondence from $E^0$ to $F^0$.
\end{prop}

\begin{remark} Let $E^0:=E^0_{(\CB,\CL,\theta,\CI_\alpha)}$, $F^0:=F^0_{(\CB,\CL,\theta,\CI_\alpha)}$ and $E^1:=E^1_{(\CB,\CL,\theta,\CI_\alpha)}$ be as above.
 Define a map 
 $d:E^1\to E^0$ by 
$	d(e^\alpha_\eta)=h_{[\alpha]\emptyset}(\eta)$. Put $$\operatorname{dom}(r):=\{e^\af_\eta:  \af \in \CL,~ \eta \cap \CR_\af \neq \emptyset\} \subset E^1,$$ where $\operatorname{dom}(r)=\bigcup_{\af \in \CL, A \in \eta \cap \CR_\af} Z^1(\af,A)$ is an open subset of $E^1$, and define a continuous map $r:\operatorname{dom}(r)\to E^0$ by $r(e^\alpha_\eta)=f_{\emptyset[\af]}(\eta)$. Then, $d$ is a local homeomorphism and  the map $\tilde{r}:E^1 \to F^0$ defined by 
$$\tilde{r}(e)=\left\{ 
\begin{array}{ll}
    r(e) & \hbox{if\ } e \in \operatorname{dom}(r), \\
    \infty & \hbox{if\ }e \notin \operatorname{dom}(r)
    \end{array}
\right.$$ is continuous by \cite[Proposition 7.1]{CasK1}. Thus, $(E^0,E^1, d,r)$ is a partially defined topological graph in the sense of  \cite[Definition 8.2]{Ka2021}.
\end{remark}

Given a topological correspondence $E=(E^1,d,r)$ from $E^0$ to $F^0$, we define the following subsets of $F^0$ (\cite[Section~1]{Ka2004}):
\begin{align*}F_{sce} & := F^0 \setminus \overline{r(E^1)}, \\
F^0_{fin}&:=\{v \in F^0: \exists V~\text{neighborhood of}~v ~\text{such that}~ r^{-1}(V)~\text{is compact} \},\\
F^0_{rg}&:=F^0_{fin} \setminus \overline{F^0_{sce}},\\
F^0_{sg} & :=F^0 \setminus F^0_{rg}.
\end{align*}
 We also consider the sets $E^0_{rg}=F^0_{rg}\cap E^0$ and $E^0_{sg}=F^0_{sg}\cap E^0$.

For $n \geq 2$, we denote by $E^n$ the space of paths of length $n$, that is,
\[E^{n}:=\{(e_1,\ldots,e_n)\in \prod_{i=1}^n E^1:d(e_i)=r(e_{i+1})(1\leq i<n)\}\]
which we regard as a subspace of the product space $\prod_{i=1}^n E^1$.  Define the {\it finite path space} $E^*= \sqcup_{n=0}^{\infty} E^n$ with the disjoint union topology. 
Define the {\it infinite path space} as
\[E^{\infty}:=\{(e_i)_{i\in\mathbb{N}}\in \prod_{i=1}^\infty E^1:d(e_i)=r(e_{i+1})(i\in\mathbb{N})\}.\]

For an element $(e_k)_{k=1}^n\in E^n$, we let $d((e_k)_{k=1}^n)=d(e_n)$ and $r((e_k)_{k=1}^n)=r(e_1)$ if $n\geq 1$. For $v\in E^0$, we let $r(v)=d(v)=v$. For infinite paths, we only define the range, namely, if $(e_k)_{k=1}^\infty\in E^\infty$, we define $r((e_k)_{k=1}^\infty)=r(e_1)$.

 We denote the length of a path $\mu \in E^* \sqcup E^\infty$ by $|\mu|$.
For convenience, we will usually write $e_1 \cdots e_n$ for $(e_1, \cdots, e_n) \in E^n$. For $e_1 \cdots e_n \in E^n$  and $\mu \in  E^* \sqcup E^\infty$, we write $e_1 \cdots e_n \mu$ for the concatenation of 
$e_1 \cdots e_n$  and $\mu$ in $ E^* \sqcup E^\infty$. For 
 $1\leq i\leq j\leq |\mu|$, we also denote by $\mu_{i,j}$ the sub-path $\mu_i\cdots \mu_j$ of  $\mu=\mu_1\mu_2\cdots\mu_{|\mu|}$, where $\mu_{i,i}=\mu_i$. If $j < i$, set $\mu_{i,j} =\emptyset$.

\begin{lem} \label{range of path} Let  $\mu=e^{\af_1}_{\eta_1} \cdots e^{\af_n}_{\eta_n} \in E^n$, where $1 \leq n$. Then, we have $$r(\mu)=f_{\emptyset[\af_1 \cdots \af_n]}\big(g_{(\af_1 \cdots \af_{n-1})\alpha_n}(\eta_n) \big).$$ 
\end{lem}

\begin{proof} We  use induction on $|\mu|=n$. The result is immediate if $|\mu|=1$. Suppose that for some integer $n \geq 2$, the result is true for paths $\mu$ with $|\mu| =n-1$.  Let  $\mu=e^{\af_1}_{\eta_1} \cdots e^{\af_{n-1}}_{\eta_{n-1}} e^{\af_n}_{\eta_n} \in E^n$ and put $\af=\af_1 \cdots \af_n$.  First,  since $h_{[\af_{n-1}]\emptyset}(\eta_{n-1})=f_{\emptyset[\af_n]}(\eta_n)$, we have $\eta_{n-1}=g_{(\af_{n-1})\emptyset} (h_{[\af_{n-1}]\emptyset}(\eta_{n-1}))=g_{(\af_{n-1})\emptyset}(f_{\emptyset[\af_n](\eta_n)})$. We then see that
\begin{align*} r(\mu)&= r(e^{\af_1}_{\eta_1} \cdots e^{\af_{n-1}}_{\eta_{n-1}}) \\
&= f_{\emptyset[\af_{1, n-1}]}\big(g_{(\af_{1, n-2})\af_{n-1}}(\eta_{n-1})\big) \\
&=f_{\emptyset[\af_{1, n-1}]}\big(g_{(\af_{1, n-2})\af_{n-1}}(g_{(\af_{n-1})\emptyset}(f_{\emptyset[\af_n]}(\eta_n)))\big) \\
&=f_{\emptyset[\af_{1, n-1}]}\big(g_{(\af_{1, n-1})\emptyset}(f_{\emptyset[\af_n]}(\eta_n))\big) \\
&= f_{\emptyset[\af_{1,n}]}(g_{(\af_{1,n-1})\alpha_n}(\eta_n)),
\end{align*}
where the second equality follows from the induction hypothesis,  the forth equality follows from \cite[Lemma 4.6 (iii)]{CasK1}, and  the last equality follows from  \cite[Lemma 4.6 (iv)]{CasK1}.
\end{proof}

\begin{dfn}(\cite[Definition 7.5]{CasK1})\label{def:boundary path space}  Let  $(\CB,\CL,\theta, \CI_\af)$ be a generalized Boolean dynamical system and $E=(E^1,d,r)$ be the associated  topological correspondence. The {\it boundary path space} of $E$ is defined by 
$$\partial E :=E^{\infty} \sqcup \{(e_k)_{k=1}^n \in E^* : d((e_k)_{k=1}^n )  \in E^0_{sg}\}.$$
We denote by $\sm_E:\partial E\setminus E^0_{sg}\to \partial E$ the shift map that removes the first edge for paths of length greater or equal to 2. For elements $\mu$ of length 1, $\sm_E(\mu)=d(\mu)$.
For a subset $S \subset E^*$, denote by 
$$\CZ(S)=\{\mu \in \partial E:~\text{either} ~r(\mu) \in S, ~\text{or there exists}~  1 \leq i \leq |\mu| ~\text{such that}~ \mu_1 \cdots \mu_i \in 
S \}.$$
We endow $\partial E$ with the topology generated by the basic open sets $\CZ(U)\cap \CZ(K)^c$, where $U$ is an open set of $E^*$ and $K$ is a compact set of $E^*$.
\end{dfn}

Note that $\partial E$ is a locally compact Hausdorff space and that $\sm_E$ is a local homeomorphism (\cite[Lemma 6.1]{KL2017}).

The following two results are frequently used throughout the paper. 

\begin{lem}\label{lem:word from path}(\cite[Lemma 7.8]{CasK1})
Let $(e^{\af_k}_{\eta_k})_{k=1}^n\in E^*$, where $1\leq n$. Then $\af_1\cdots\af_n\in\CW^*$. Moreover for all $1\leq m\leq n$, we have that $\eta_m\cap \CI_{\af_{1,m}}$ is an ultrafilter in $\CI_{\af_{1,m}}$.
\end{lem}

\begin{thm}\label{thm:iso.tight.boundary}(\cite[Theorem 7.10]{CasK1})
Let $(\CB,\CL,\theta,\CI_\af)$ a generalized Boolean dynamical system, $\mathsf{T}$ the tight spectrum of its inverse semigroup and $\partial E$ the boundary path space of its topological correspondence. Then, there exists a homeomorphism $\phi:\mathsf{T}\to \partial E$ defined by 
\[\phi(\xi^\af)=\begin{cases}
\xi_0 & \text{if } \af=\emptyset, \\
(e^{\af_n}_{\eta_n})_{n=1}^{|\af|} &\text{if }|\af|\geq 1,
\end{cases}\]
where $\eta_1=\xi_1$ and $\eta_n=h_{[\af_{1,n-1}]\af_n}(\xi_n)$ for $2\leq n\leq|\af|$. 
\end{thm}

\begin{remark}\label{basis of partial E}
It follows immediately from Theorem \ref{thm:iso.tight.boundary} and Remark \ref{basis of T}, that $\partial E$ has a basis of compact-open sets.
\end{remark}

By Lemma \ref{lem:word from path},  we have a map $\CP:\partial E\to\CW^{\leq\infty}$ that sends a path $(e^{\af_k}_{\xi_k})_{k=1}^N$, where $1\leq N\leq\infty$, to $(\af_k)_{k=1}^N$ and an element of $E^0$ to $\emptyset$.
For $\alpha=\af_1 \cdots \af_{|\af|} \in\mathcal{W}^{\geq 1}$ and   $A  \in \CI_\af$, we then define the cylinder set by 
\begin{align*}\CN(\af,A)&:=\{(e_{\eta_n}^{\beta_n}) \in \partial E: \CP((e_{\eta_n}^{\beta_n}))_{1,|\af|}=\af ~\text{and}~   A\in \eta_{|\alpha|}\}\\
&=\{(e_{\eta_n}^{\beta_n}) \in \partial E: \bt_1 \cdots \bt_{|\af|}=\af ~\text{and}~   A\in \eta_{|\alpha|}\}.
\end{align*}
Also, for  $A\in\CB$, we define 
\begin{align*}\CN(\emptyset,A):=\{\mu \in\partial E: A\in r(\mu)\} .
\end{align*}
Note that if $\af\in\CW^*$ and $A=\emptyset$, then $\CN(\af,A)=\emptyset$.

\begin{lem}\label{compact open set} For $\alpha\in\mathcal{W}^*$ and $A \in \CI_\af$,  the sets $\CN(\af,A)$
are compact-open sets in $\partial E$.
\end{lem}

\begin{proof} Let $\phi: \mathsf{T} \to \partial E$ be the homeomorphism defined in Theorem \ref{thm:iso.tight.boundary}. Then for $\af \in \mathcal{W}^*$ and $A \in \CI_\af \setminus\{\emptyset\}$, one can see that $\phi(V_{(\af,A,\af)})= \CN(\af,A)$. Thus, $\CN(\af,A)$ is a  compact open set in $\partial E$. If $A=\emptyset$, then $\CN(\af,A)=\emptyset$ is compact-open.
 \end{proof}

\begin{lem}\label{basis} For $\af, \bt \in \CW^*$ and $A \in \CI_\af \setminus \{\emptyset\}, B \in \CI_\bt \setminus \{\emptyset\}$, we have 
$$\CN(\af,A) \cap \CN(\bt,B)=
\left\{
   \begin{array}{ll}
    \CN(\af, \theta_{\af'}(B) \cap A)   & \hbox{if\ }  \af=\bt\af',   \\
    \CN(\bt, \theta_{\bt'}(A) \cap B)  & \hbox{if\ } \bt=\af\bt', \\
     \CN(\af, A \cap B) & \hbox{if\ }  \af=\bt,  \\
      \emptyset & \hbox{otherwise.}
   \end{array}
\right.$$ 
\end{lem}

\begin{proof} If $ \af=\bt\af'$, then for 
$\mu =(e^{\mu_i}
_{\eta_i}) \in \partial E$ such that $\CP(\mu)_{1,|\af|}=\af$,  we have  
$$h_{[\bt_{|\bt|}]\emptyset}(\eta_{|\bt|})=f_{\emptyset[\af']}(g_{(\af_{1, |\af'|-1})\af'_{|\af'|}}(\eta_{|\bt\af'|}))$$
by Lemma \ref{range of path}. 
So, for $B \in \CI_\bt$, one can see that $B \in \eta_{|\bt|} \iff \theta_{\af'}(B) \in \eta_{|\bt\af'|}. $
Then, since $\eta_{|\af|}$ is a filter, $A, \theta_{\af'}(B) \in \eta_{|\af|}  $ if and only if   $\theta_{\af'}(B) \cap A \in \eta_{|\af|} $. Thus, we can conclude that $\CN(\af,A) \cap \CN(\bt,B)=\CN(\af, \theta_{\af'}(B) \cap A) $.  Similarly, we have the others.
\end{proof}

\begin{lem}\label{union of open sets}
For $\af\in\CW^*$ and $A_1,\ldots,A_n \in \CI_\af \setminus \{\emptyset\}$, we have 
$$\bigcup_{i=1}^n\CN(\af,A_i)= \CN\left(\af,\bigcup_{i=1}^n A_i\right).$$
\end{lem}

\begin{proof}
This follows immediately from the fact that for an ultrafilter $\eta$ of $\CI_{\af_{|\af|}}$, we have that $\bigcup_{i=1}^n A_i\in\eta$ if and only if $A_i\in\eta$ for some $i=1,\ldots,n$.
\end{proof}

\subsection{A partial action on $\partial E$}
Let $\mathbb{F}$ be the free group generated by $\CL$. We identify the identity of $\mathbb{F}$ with $\emptyset$ and note that we can see $\CW^*$ as a subset of $\mathbb{F}$. 
To define a partial action of $\mathbb{F}$ on $\partial E$, we first
define  the following sets:

\begin{itemize}
\item $U_\emptyset=\partial E$;
\item for $\af, \bt \in \CW^{\geq 1}$  such that $\CI_\af \cap \CI_\bt \neq \emptyset$, let
\[ U_{\af\bt^{-1}}  =\{(e^{\gm_k}_{\eta_k})_{k = 1}^{N}\in\partial E :\gm_1 \cdots \gm_{|\af|}=\af \text{ and }\eta_{|\af|}\cap\CI_{\bt}\neq\emptyset \};\]
 \item for $\af \in \CW^{\geq 1}$, let \[U_{\af}:=U_{\af\emptyset}=\{(e^{\gm_k}_{\eta_k})_{k=1}^N\in\partial E : \gm_1 \cdots \gm_{|\af|}=\af \};\]
   \item for $\bt   \in \CW^{\geq 1}$, let \[ U_{\bt^{-1}}:=U_{\emptyset \bt^{-1}} =\{\mu \in\partial E : r(\mu) \cap \CI_\bt \neq \emptyset\};\]
\item for all the other elements $\gm \in \mathbb{F}$, let $U_\gm =\emptyset$. 
   \end{itemize}

\begin{lem} \label{open sets} For $\af, \bt \in \CW^{*}$, the set $U_{\af\bt^{-1}}$ is an open set in $\partial E$.
\end{lem}
 
 \begin{proof} $U_{\emptyset}=\partial E$ is open. For $\af,\bt \in \CW^{\geq 1}$, it is easy to see that   $U_{\af\bt^{-1}} = \bigcup_{A \in \CI_\bt} \CN(\af, A)$,  $U_\af=U_{\af \emptyset}= \bigcup_{A \in \CB} \CN(\af, A) $, and   $U_{\beta^{-1}}=U_{\emptyset\beta^{-1}}=\bigcup_{A\in\CI_{\beta}}\CN(\emptyset,A)$. Thus, $U_{\af\bt^{-1}}$  is an open set in $\partial E$ for $\af, \bt \in \CW^{*}$. 
 \end{proof}
 
 \begin{lem}\label{U af closed}
 For $\af\in\CW^*$, $U_\af$ is closed in $\partial E$.
 \end{lem}
 
 \begin{proof} 
  $U_{\emptyset}=\partial E$ is closed. For $\af\in\CW^{\geq 1}$, we show that $\partial E\setminus U_{\af}$ is open. Let $\mu\in \partial E\setminus U_{\af}$ and $\beta=\CP(\mu)$. Suppose first that there exists $1\leq i\leq |\af|$ such that $\beta_i\neq\af_i$. Recall that $\mu_i=e^{\bt_i}_{\eta_i}$ for some  ultrafilter $\eta_i$ in $\CI_{\bt_i}$. In this case for $B\in\eta_i$, we have that $\mu\in\CN(\beta_{1,i},B)\subseteq\partial E\setminus U_{\af}$. The second and final case is when $\beta=\af_{1,n}$ for some $0\leq n<|\af|$. Let $B$ be such that $\mu\in\CN(\af_{1,n},B)$. By Lemma \ref{compact open set}, $\CN(\af_{1,n},B)\setminus\CN(\af_{1,n+1},\theta_{\af_{n+1}}(B))$ is open and note that $\mu\in \CN(\af_{1,n},B)\setminus\CN(\af_{1,n+1},\theta_{\af_{n+1}}(B))\subseteq \partial E\setminus U_{\af}$. In all cases, we have proved that $\mu$ is an interior point of $\partial E\setminus U_{\af}$ so that $U_{\af}$ is closed.
 \end{proof}

\begin{lem}\label{adding a path} Let $\xi \in X_{\emptyset}$ such that  $\xi \cap \CI_\af \neq \emptyset$ for some $\af \in \CW^*$. Define 
\begin{align*} \eta_n &:=\xi \cap \CI_{\af_n} , \\
\eta_{n-1}&:=f_{\emptyset[\af_n]}(\eta_n) \cap \CI_{\af_{n-1}}  , \\
\eta_{n-2}&:=f_{\emptyset[\af_{n-1}]}(\eta_{n-1}) \cap \CI_{\af_{n-2}}, \\
& \vdots \\
\eta_2 &:=f_{\emptyset[\af_3]}(\eta_3) \cap \CI_{\af_2},\\
\eta_1&:=f_{\emptyset[\af_2]}(\eta_2) \cap \CI_{\af_1}. 
\end{align*}
Then we have  $\emptyset \neq \eta_i \in X_{\af_i}$ for  $i=1, \cdots, n$.
\end{lem}

\begin{proof} Since $\xi \cap \CI_\af \neq \emptyset$, we have $\theta_{\af_{2,n}}(A) \in \xi $ for some $A \in \CI_{\af_1}$. Thus, $ \theta_{\af_n}(\theta_{\af_{2,n-1}}(A)) \in \xi \cap \CI_{\af_n} = \eta_n $. Thus, $ \emptyset \neq \eta_n \in X_{\af_n}$ by \cite[Proposition 4.2]{CasK1}. 
From  $ \theta_{\af_n}(\theta_{\af_{2,n-1}}(A)) \in \eta_n $, it also follows that $\theta_{\af_{2,n-1}}(A) =\theta_{\af_{n-1}}(\theta_{\af_{2,n-2}}(A)) \in f_{\emptyset[\af_n]}(\eta_n) \cap \CI_{\af_{n-1}}$. Thus, $\emptyset \neq \eta_{n-1} \in X_{\af_{n-1}} $. Continuing this process, we have 
that  $\theta_{\af_{2,i}}(A)= \theta_{\af_{i}}(\theta_{\af_{2,i-1}}(A)) \in f_{\emptyset[\af_{i+1}]}(\eta_{i+1})\cap\CI_{\af_{i}} (= \eta_i)$ for $i=1, \cdots, n-2$. Thus, $\emptyset \neq \eta_i \in X_{\af_i}$ for all $i=1, \cdots, n$.
\end{proof}

We now  define the following maps: 
\begin{itemize}
\item  for the identity $\emptyset$ of $\mathbb{F}$, we define $\varphi_\emptyset: U_\emptyset \to U_\emptyset$ as the identity map;
\item  for $\af\in \CW^{\geq 1}$ and $\beta\in\CW^*$ such that $\alpha\beta^{-1}$ is in reduced form in $\mathbb{F}$ and $U_{\bt\af^{-1}}\neq\emptyset$, we define a map $\varphi_{\af\bt^{-1}}:U_{\bt\af^{-1}}\to U_{\af\bt^{-1}}$ that sends a path
$\mu =(e^{\mu_i}_{\xi_i})_{i = 1}^N\in U_{\bt\af^{-1}}$ that removes the first $|\bt|$ coordinates of $\mu$ and adds at the beginning $(e^{\af_1}_{\eta_1},\ldots,e^{\af_n}_{\eta_{n}})$ ($n=|\af|$), where
\begin{align*} \eta_n &:=g_{(\af_n)\emptyset}(h_{[\bt_{|\bt|}]\emptyset}(\xi_{|\bt|})) , \\
\eta_{n-1}&:=g_{(\af_{n-1})\emptyset}(f_{\emptyset[\af_n]}(\eta_n)), \\
\eta_{n-2}&:=g_{(\af_{n-2})\emptyset}(f_{\emptyset[\af_{n-1}]}(\eta_{n-1})), \\
& \vdots \\
\eta_2 &:= g_{(\af_2)\emptyset}(f_{\emptyset[\af_3]}(\eta_3)),\\
\eta_1&:= g_{(\af_1)\emptyset}(f_{\emptyset[\af_2]}(\eta_2)), 
\end{align*}
 that is, 
\[\varphi_{\af\bt^{-1}}(\mu):=e^{\af_1}_{\eta_1}\ldots e^{\af_n}_{\eta_{n}}\mu_{|\bt|+1,|\mu|};\]
\item  for $\bt   \in \CW^{\geq 1}$, we define a map  $\varphi_{\bt^{-1}} : U_{\bt} \to U_{\bt^{-1}}$ that  removes the first $|\beta|$ coordinates; 
\item for all the other elements $\gm \in \mathbb{F}$, define $\varphi_\gm: U_{\gm^{-1}} \to U_\gm$ as the empty map. 
\end{itemize}

We then see in the following Lemma, that for    $\af,\bt \in \CW^*$, 
$\varphi_{\af\bt^{-1}}(\mu)$ is a well-defined path on the topological correspondence $(E^1, d,r)$  from $E^0$ to $F^0$ and  $\varphi_{\af\bt^{-1}}(\mu) \in  U_{\af\bt^{-1}}$.

\begin{lem} Let $\eta_i$ be defined as in the second point above for  $i=1, \cdots, n$. Then we have 
\begin{enumerate}
\item[(i)] $e^{\af_i}_{\eta_i} \in E^1$  for  $i=1, \cdots, n$,
\item[(ii)]  $d(e^{\af^{i-1}}_{\eta_{i-1}})=r(e^{\af^{i}}_{\eta_{i}})$ for  $i=2, \cdots, n$, and $d(e^{\af_n}_{\eta_n})=r(e^{\mu^{|\bt|+1}}_{\xi_{|\bt|+1}})$,
\item[(iii)]  $\eta_n \cap\CI_\bt \neq  \emptyset$. 
\end{enumerate}
\end{lem}

\begin{proof}(i) First note that if $\mu =(e^{\mu_i}_{\xi_i})_{i = 1}^N\in U_{\bt\af^{-1}}$, then $\xi_{|\bt|} \cap \CI_\af \neq \emptyset$. So, $ h_{[\bt_{|\bt|}]\emptyset}(\xi_{|\bt|}) \cap \CI_\af \neq \emptyset$. 
Now, by the definition of the map $g_{(\af)\emptyset}$ for $\af \in \CL$, we see that  $$\eta_n =g_{(\af_n)\emptyset}(h_{[\bt_{|\bt|}]\emptyset}(\xi_{|\bt|}))= h_{[\bt_{|\bt|}]\emptyset}(\xi_{|\bt|}) \cap \CI_{\af_n}$$ and 
$$\eta_i = g_{(\af_i)\emptyset}(f_{\emptyset[\af_{i+1}]}(\eta_{i+1}))=f_{\emptyset[\af_{i+1}]}(\eta_{i+1}) \cap \CI_{\af_i}$$ for all $i=1, \cdots, n-1$. Then,  by Lemma \ref{adding a path},  $\emptyset \neq \eta_i \in X_{\af_i}$ $i=1, \cdots, n$. So, $e^{\af_i}_{\eta_i} \in E^1$  for  $i=1, \cdots, n$.

(ii) Since $g_{(\af_{i-1})\emptyset}$ and $h_{[\af_{i-1}]\emptyset}$ are mutually inverse for each $i=2, \cdots, n$, we have  $$d(e^{\af^{i-1}}_{\eta_{i-1}})= h_{[\af_{i-1}]\emptyset}(\eta_{i-1})=h_{[\af_{i-1}]\emptyset}(g_{(\af_{i-1})\emptyset}(f_{\emptyset[\af_i]}(\eta_i)))=f_{\emptyset[\af_i]}(\eta_i)=r(e^{\af^{i}}_{\eta_{i}})$$
for  $i=2, \cdots, n$.
  Also, since $g_{(\af_{n})\emptyset}$ and $h_{[\af_{n}]\emptyset}$ are mutually inverse, we have \begin{align*}\label{eq:1}d(e^{\af_n}_{\eta_n})=h_{[\af_n]\emptyset}(\eta_n)=h_{[\af_n]\emptyset}(g_{(\af_n)\emptyset}(h_{[\bt_{|\bt|}]\emptyset}(\xi_{|\bt|}))=h_{[\bt_{|\bt|}]\emptyset}(\xi_{|\bt|})=r(e^{\mu^{|\bt|+1}}_{\xi_{|\bt|+1}}).\end{align*}

  (iii) Since $\xi_{|\bt|}\cap \CI_\bt \neq \emptyset$ by Lemma \ref{lem:word from path}, $h_{[\bt_{|\bt|}]\emptyset}(\xi_{|\bt|}) \cap \CI_\bt \neq \emptyset $.
Also, since $\xi_{|\bt|}\cap \CI_\af \neq \emptyset$, we have $h_{[\bt_{|\bt|}]\emptyset}(\xi_{|\bt|}) \cap \CI_\af \neq \emptyset $. Note that if $A\in h_{[\bt_{|\bt|}]\emptyset}(\xi_{|\bt|}) \cap \CI_\af$ and $B\in h_{[\bt_{|\bt|}]\emptyset}(\xi_{|\bt|}) \cap \CI_\bt$, then $A\cap B\in h_{[\bt_{|\bt|}]\emptyset}(\xi_{|\bt|}) \cap \CI_\af\cap\CI_\bt$. 
Then, since $h_{[\bt_{|\bt|}]\emptyset}(\xi_{|\bt|}) \cap \CI_\af  \cap \CI_\bt\neq \emptyset $ and $\CI_\af \subset \CI_{\af_n}$, it follows that 
 $\eta_n \cap \CI_\bt = \big( h_{[\bt_{|\bt|}]\emptyset}(\xi_{|\bt|}) \cap \CI_{\af_n} \big)\cap \CI_\bt \neq \emptyset$. 
\end{proof}

\begin{prop}\label{homeo} For each $\af \in \CL$, the maps $\varphi_\af: U_{\af^{-1}} \to U_\af$  and $\varphi_{\af^{-1}}: U_\af \to U_{\af^{-1}}$ are homeomorphisms with $\varphi_\af^{-1}=\varphi_{\af^{-1}}$.
\end{prop}

\begin{proof} Fix $\af \in \CL$. For $\mu=(e^{\mu_i}_{\xi_i})_{i=1}^N \in U_{\af^{-1}}$, 
we have 
$\varphi_\af(\mu)=e^{\af}_{\eta}\mu$, where $\eta=g_{(\af)\emptyset}(f_{\emptyset[\mu_1]}(\xi_1))=f_{\emptyset[\mu_1]}(\xi_1) \cap \CI_\af$, and for $\nu=(e^{\nu_i}_{\chi_i})_{i=1}^{N'} \in U_\af$, we have $\varphi_{\af^{-1}}(\nu)=(e^{\nu_i}_{\chi_i})_{i=2}^{N'}$. 
Thus, $\varphi_{\af^{-1}}  \varphi_{\af}(\mu)=\varphi_{\af^{-1}}(e^{\af}_{\eta}\mu)=\mu$
for each $\mu \in U_{\af^{-1}}$. Also, we see that
$$ \varphi_{\af} \varphi_{\af^{-1}}(\nu)= \varphi_{\af} ((e^{\nu_i}_{\chi_i})_{i=2}^{N'})=e^\af_{\xi}(e^{\nu_i}_{\chi_i})_{i=2}^{N'},$$
where $\xi=g_{(\af)\emptyset}(f_{\emptyset[\nu_2]}(\chi_2))$. On the other hand, 
Since $h_{[\af]\emptyset}(\chi_1)=f_{\emptyset[\nu_2]}(\chi_2)$, it follows that $
\chi_1=g_{(\af)\emptyset}(f_{\emptyset[\nu_2]}(\chi_2))=\xi$. 
Thus, $e^\af_{\xi}(e^{\nu_i}_{\chi_i})_{i=2}^{N'} =\nu$.
So, we conclude that 
 $\varphi_{\af^{-1}}  \varphi_{\af}= id_{U_{\af^{-1}}}$ and $ \varphi_{\af} \varphi_{\af^{-1}}=id_{U_{\af}} $, and hence,   $\varphi_{\af^{-1}}=\varphi_\af^{-1}$ for each $\af \in \CL$.
   
 Note that $\varphi_{\af^{-1}}=\sm_E|_{U_\af}$, where $\sm_E$ is as in Definition \ref{def:boundary path space}. Since $\sm_E$ is a local homeomorphism, and $\varphi_{\af^{-1}}$ is a bijection between open subsets, we have that $\varphi_{\af^{-1}}$ is a homeomorphism, and so is its inverse $\varphi_{\af}$.
\end{proof}

\begin{prop}\label{prop:partial action} $\Phi=(\{U_t\}_{t \in \mathbb{F}}, \{\varphi_t\}_{t \in \mathbb{F}})$ is a semi-saturated orthogonal  partial action of  $\mathbb{F}$ on $\partial E$.
\end{prop}

\begin{proof} Let $\af=\af_1 \cdots \af_n \in \CW^{\geq 1}$ and  $\mu=(e^{\mu_i}_{\xi_i}) \in U_{\af^{-1}}$.  We then have $\varphi_\af(\mu)=e^{\af_1}_{\eta_1}\ldots e^{\af_n}_{\eta_{n}} \mu$, where
 $$\eta_n =g_{(\af_n)\emptyset}(f_{\emptyset[\mu_1]}(\xi_1))$$ and 
$$\eta_i = g_{(\af_i)\emptyset}(f_{\emptyset[\af_{i+1}]}(\eta_{i+1}))$$ for all $i=1, \cdots, n-1$.
Also, we see that $$\varphi_{\af_1} \ldots \varphi_{\af_n}(\mu)=\varphi_{\af_1} \ldots \varphi_{\af_{n-1}}(e^{\af_n}_{\eta_n}\mu)=\cdots = \varphi_{\af_1}(e^{\af_2}_{\eta_2} \ldots e^{\af_n}_{\eta_{n}} \mu)=e^{\af_1}_{\eta_1}\ldots e^{\af_n}_{\eta_{n}} \mu.$$
We prove by induction that $\text{dom}(\varphi_{\af_1} \circ \cdots  \circ \varphi_{\af_n})=U_{\af^{-1}}$. This is immediate for $n=1$. For $n>1$, suppose that $\text{dom}(\varphi_{\af_2} \circ \cdots  \circ \varphi_{\af_n})=U_{\af_{2,n}^{-1}}$. On the one hand, if $\mu\in U_{\alpha^{-1}}$, then $\varphi_{\af_2} \circ \cdots  \circ \varphi_{\af_n}(\mu)\in\text{dom}(\varphi_{\alpha_1})$ by Lemma \ref{adding a path}. 
On the other hand if $\mu\in U_{\af_{2,n}^{-1}}$ and $\nu:=\varphi_{\af_2} \circ \cdots  \circ \varphi_{\af_n}(\mu)\in\text{dom}(\varphi_{\alpha_1})$, then $r(\nu)\cap \CI_{\af_1} \neq \emptyset$. We also have
\begin{align*}
    r(\nu)&=r(e^{\af_2}_{\eta_2} \ldots e^{\af_n}_{\eta_{n}} \mu)\\
    &=r(e^{\af_2}_{\eta_2} \ldots e^{\af_n}_{\eta_{n}})\\
    &=f_{\emptyset[\af_{2,n}]}\big(g_{(\af_{2,n-1})\alpha_n}(\eta_n) \big)\\
    &=f_{\emptyset[\af_{2,n}]}\big(g_{(\af_{2,n-1})\alpha_n}\big(g_{(\af_n)\emptyset}(f_{\emptyset[\mu_1]}(\xi_1))\big) \big)\\
    &=f_{\emptyset[\af_{2,n}]}\big(g_{(\af_{2,n})\emptyset}(r(\mu)) \big),
\end{align*}
where the third equality follows from Lemma \ref{range of path} and the last equality follows from \cite[Lemma 4.6(iii)]{CasK1} and the definition of $r(\mu)$.
So, there is $A\in  r(\nu)\cap \CI_{\af_1}$ such that $\theta_{\alpha_{2,n}}(A)\in r(\mu)\cap \CI_{\alpha_{2,n}}$. Since $A\in \CI_{\af_1}$, we have $\theta_{\alpha_{2,n}}(A)\in \CI_{\af}\cap r(\mu)$, and hence $\mu \in U_{\af^{-1}}$.
Thus, we have  $\varphi_\af=\varphi_{\af_1} \circ \cdots  \circ \varphi_{\af_n}$.

Let $\bt=\bt_1 \cdots \bt_m \in \CW^{\geq 1}$ and $\mu=(e^{\mu_i}_{\xi_i})_{i \geq 1} \in U_\bt$. Then, $\varphi_{\bt^{-1}}(\mu)=(e^{\mu_i}_{\xi_i})_{i \geq |\bt|+1}$. 
Also, $$\varphi_{\bt_m^{-1}} \ldots \varphi_{\bt_1^{-1}}(\mu)=\varphi_{\bt_m^{-1}} \ldots \varphi_{\bt_2^{-1}}((e^{\mu_i}_{\xi_i})_{i \geq 2})=\cdots=\varphi_{\bt_m^{-1}}((e^{\mu_i}_{\xi_i})_{i \geq |\bt|}) =(e^{\mu_i}_{\xi_i})_{i \geq |\bt|+1}. $$
It is clear that $\text{dom}(\varphi_{\bt_m^{-1}} \circ \ldots \circ \varphi_{\bt_1^{-1}})=U_{\beta}$. Thus, we see that $\varphi_{\bt^{-1}}=\varphi_{\bt_m^{-1}} \circ \ldots \circ \varphi_{\bt_1^{-1}}$.

Now, for $g=\af\bt^{-1} \in \mathbb{F}$ in reduced form, where $\af=\af_1 \cdots \af_n, \bt=\bt_1 \cdots \bt_m \in \CW^{\geq 1}$, as above we see that 
 $$\varphi_{g}= \varphi_{\af_1} \circ \cdots  \circ \varphi_{\af_n} \circ \varphi_{\bt_m^{-1}} \circ \ldots \circ \varphi_{\bt_1^{-1}} .$$ Note also that for $\af,\bt\in\CL$, with $\af\neq\bt$, $\varphi_{\af^{-1}}\varphi_{\bt}$ is the empty function. We have proved that if $s=x_1\cdots x_n\in\mathbb{F}$ is in reduced form, then $\varphi_s=\varphi_{x_1}\circ\cdots\circ\varphi_{x_n}$. Thus, by \cite[Propositions 4.7 and 4.10]{ExelBook},  $\Phi=(\{U_t\}_{t \in \mathbb{F}}, \{\varphi_t\}_{t \in \mathbb{F}})$ is a partial action of  $\mathbb{F}$ on $\partial E$. 
 
 For $\af, \bt \in \CL$, if $\af \neq \bt$, then we clearly see that $U_\af \cap U_\bt = \emptyset$. So, the action is orthogonal. 
\end{proof}

\subsection{The partial crossed product  $C_0(\partial E) \rtimes_{\hat{\varphi}} \mathbb{F}$} Let $(\CB,\CL,\theta,\CI_\af)$ be a generalized Boolean dynamical system and let $\Phi=(\{U_t\}_{t \in \mathbb{F}}, \{\varphi_t\}_{t \in \mathbb{F}})$ be the partial action of  $\mathbb{F}$ on $\partial E$ associated with $(\CB,\CL,\theta,\CI_\af)$. We in this section how to describe the partial crossed product $C^*$-algebra $C_0(\partial E) \rtimes_{\hat{\varphi}} \mathbb{F}$.
For $t \in \mathbb{F}$, the set $U_{t}$ is an open set in $\partial E$.
 So, any $f \in C_0(U_t)$ can  be viewed as a function in $C_0(\partial E)$ by declaring that $f(\mu)=0$ if $\mu \notin U_t$.
In fact, $C_0(U_t)$  is a closed two-sided ideal in $C_0(\partial E)$ and thus a $C^*$-subalgebra. 

For $t \in \mathbb{F}$, put 
$$D_t=C_0(U_t) ~\text{and}~D_{t^{-1}}=C_0(U_{t^{-1}}).$$
Define $\hat{\varphi}_t: D_{t^{-1}} \to D_t$ by $$\hat{\varphi}_t(f)=f \circ \varphi_{t^{-1}}.$$
Then, $(\{D_t\}_{t \in \mathbb{F}}, \{\hat{\varphi}_t\}_{t \in \mathbb{F}} )$ is a $C^*$-algebraic partial dynamical system. Hence, we may consider  the partial crossed product 
$$C_0(\partial E) \rtimes_{\hat{\varphi}} \mathbb{F} = \overline{\operatorname{span}} \Big\{ \sum_{t \in \mathbb{F}}f_t\dt_t: f_t \in D_t ~\text{and}~ f_t \neq 0  ~\text{for finitely many}~ t \in \mathbb{F} \Big\},$$
where the closure is with respect to the universal norm. Note that $\dt_t$ has no meaning in itself and merely serves a place holder. Recall that multiplication and involution in $C_0(\partial E) \rtimes_{\hat{\varphi}} \mathbb{F}$ are given by 
\begin{align*} &(a\dt_s)(b \dt_t)=\hat{\varphi}_s(\hat{\varphi}_{s^{-1}}(a)b)\dt_{st}, ~\text{and}~ \\
&(a\dt_s)^*=\hat{\varphi}_{s^{-1}}(a)\dt_{s^{-1}}. 
\end{align*}

For $\af \in \CW^*$ and $ A \in \CI_\af$, we let $1_{\CN(\af,A)}$ denote the characteristic function on $\CN(\af,A)$. We first show that $C_0(\partial E)$ is generated by these characteristic functions.

\begin{lem}\label{C_0} The $C^*$-algebra $C_0(\partial E)$ is generated by the set 
$$\{1_{\CN(\af,A)}: \af \in \CW^* ~\text{and}~ A \in \CI_\af\}.$$
\end{lem}

\begin{proof} Choose $\mu=(e^{\mu_i}_{\xi_i})_{1 \leq i \leq N} \in \partial E$, with $|\mu|\geq 1$. Then $\mu_1\in \CW^*$ and $\xi_1 \cap \CI_{\mu_1} (\neq \emptyset)$ 
is an ultrafilter in $\CI_{\mu_1}$ by Lemma \ref{lem:word from path}.
 Choose $\emptyset \neq A \in \xi_1 \cap \CI_{\mu_1}$. Then, we have $\mu \in \CN(\mu_1, A)$, and hence, $1_{\CN(\mu_1, A)}(\mu)=1$. On the other hand for $\mu\in\partial E\cap E^0$, $\mu$ is an ultrafilter and for $A\in\mu$, we have $\mu\in\CN(\emptyset,A)$, so that $1_{\CN(\emptyset, A)}(\mu)=1$.
So, the set $\{1_{\CN(\af,A)}: \af \in \CW^* ~\text{and}~ A \in \CI_\af\}$ vanishes nowhere. 

Let $\mu \neq \nu \in \partial E$ and say $\mu=(e^{\mu_i}_{\xi_i})$ and 
$\nu=(e^{\nu_i}_{\eta_i})$. If $\CP(\mu) \neq \CP(\nu)$, then the two points are clearly separated by the set $\{1_{\CN(\af,A)}: \af \in \CW^* ~\text{and}~ A \in \CI_\af\}$. 
If $\CP(\mu)=\CP(\nu)$, then $\mu_i=\nu_i$ for all $i \geq 1$ and there is $n \geq 1$ such that $\xi_n \neq \eta_n$.   
Since $\xi_n=h_{[\mu_1 \cdots \mu_{n-1}]\mu_n}(g_{(\mu_1 \cdots \mu_{n-1})\mu_n}(\xi_n))= \uparrow_{\CI_{\mu_n}} \xi_n \cap \CI_{\mu_1 \cdots \mu_n}   $ and also $\eta_n =\uparrow_{\CI_{\mu_n}} \eta_n \cap \CI_{\mu_1 \cdots \mu_n}  $, it follows  that $\xi_n \cap \CI_{\mu_1 \cdots \mu_n} \neq \eta_n \cap \CI_{\mu_1 \cdots \mu_n}$. 
Choose $A \in \CB$ such that $A \in \xi_n \cap \CI_{\mu_1 \cdots \mu_n} $, but $A \notin \eta_n \cap \CI_{\mu_1 \cdots \mu_n} $.
Then
$$1_{\CN(\mu_1 \cdots \mu_n, A)}(\mu)=1 \neq 0 = 1_{\CN(\mu_1 \cdots \mu_n, A)}(\nu),$$
which shows that the set $\{1_{\CN(\af,A)}: \af \in \CW^* ~\text{and}~ A \in \CI_\af\}$ separates points. So, the Stone-Weierstrass Theorem gives the result. 
\end{proof}

\begin{lem}\label{computations} Let $C^*(\{1_{\CN(\emptyset,A)}\dt_\emptyset, 1_{\CN(\af,B)}\dt_\af\}) \subseteq C_0(\partial E) \rtimes_{\hat{\varphi}} \mathbb{F} $ denote the $C^*$-subalgebra generated by $\{1_{\CN(\emptyset,A)}\dt_\emptyset, 1_{\CN(\af,B)}\dt_\af: A \in \CB, \af \in \CL ~\text{and}~ B \in \CI_\af \}$. Then, we have, 
for $\af=\af_1 \cdots \af_n \in \CW^{\geq 1}$ and $A \in \CI_\af$ is such that $A \subseteq \theta_{\af_2 \cdots \af_n}(B)$ for some   $B \in \CI_{\af_1}$, 
\begin{enumerate}
\item[(i)]  $ \hat{\varphi}_{\af^{-1}}(1_{\CN(\af,A)})=1_{\CN(\emptyset,A)} ~\text{and}~ \hat{\varphi}_{\af}(1_{\CN(\emptyset,A)})=1_{\CN(\af,A)},$
\item[(ii)] 
$1_{\CN(\af,A)}\dt_\af=(1_{\CN(\af_1, B)}\dt_{\af_1})( 1_{\CN(\af_2, \theta_{\af_2}(B))}\dt_{\af_2})(1_{\CN(\af_3, \theta_{\af_2\af_3}(B))}\dt_{\af_3})\cdots (1_{\CN(\af_n, A)}\dt_{\af_n}),$ 
\item [(iii)] $(1_{\CN(\af,A)}\dt_\af)^*=1_{\CN(\emptyset,A)}\dt_{\af^{-1}}$,
\item[(iv)] $(1_{\CN(\af,A)}\dt_\af)(1_{\CN(\af,A)}\dt_\af)^*=1_{\CN(\af,A)}\dt_\emptyset $, and 
\item[(v)]for $\af,\bt \in \CW^{\geq 1}$ such that $\af\bt^{-1} \in \mathbb{F}$ is in reduced form and $A \in \CI_\af \cap \CI_\bt$,  $$1_{\CN(\af,A)}\dt_{\af\bt^{-1}}= (1_{\CN(\af,A)}\dt_\af)(1_{\CN(\emptyset, A)}\dt_{\bt^{-1}}),$$
\item[(vi)]  for all $A \in \CB$, $\af \in \CW^{\geq 1}$ and $B \in \CI_\af$, 
$$(1_{\CN(\emptyset,A)}\dt_\emptyset)(1_{\CN(\af,B)}\dt_\af)=(1_{\CN(\af,B)}\dt_\af)(1_{\CN(\emptyset,\theta_\af(A))}\dt_\emptyset).$$
\end{enumerate}
\end{lem}

\begin{proof}(i) Let $\af \in \CW^{\geq 1}$ and $A \in \CI_\af$. We first claim that  
 $$\varphi_{\af^{-1}}(\CN(\af,A))=\CN(\emptyset,A).$$
For $\mu=(e^{\mu_i}_{\xi_i})_{i \geq 1} \in \CN(\af,A)$, since  $A \in \xi_{|\af|}$, we have $ A \in h_{[\mu_{|\af|}]
\emptyset}(\xi_{|\af|})= f_{\emptyset[\mu_{|\af|+1}]}(\xi_{|\af|+1})$.  Thus, $\varphi_{\af^{-1}}(\mu)=(e^{\mu_i}_{\xi_i})_{i \geq |\af|+1} \in \CN(\emptyset,A)$. Hence,  $\varphi_{\af^{-1}}(\CN(\af,A)) \subseteq \CN(\emptyset,A)$.
For the converse, note first that $\CN(\emptyset, A) \subset U_{\af^{-1}}$. 
If $\mu \in \CN(\emptyset,A)$, then $\varphi_{\af}(\mu)=e^{\af_1}_{\eta_1} \cdots e^{\af_n}_{\eta_n}\mu$, where $\eta_n=r(\mu)\cap \CI_{\af_n}, \eta_{n-1}=f_{\emptyset[\af_n]}(\eta_n) \cap \CI_{\af_{n-1}}  , \cdots, \eta_1=f_{\emptyset[\af_2]}(\eta_2) \cap \CI_{\af_1}$. 
Since $A \in r(\mu) \cap \CI_\af \subset r(\mu) \cap \CI_{\af_n} $, we have $\varphi_{\af}(\mu) \in \CN(\af,A)$. 
So,   $\mu=\varphi_{\af^{-1}}(\varphi_\af(\mu)) \in \varphi_{\af^{-1}}(\CN(\af,A))$.
Hence, $ \CN(\emptyset,A) \subseteq \varphi_{\af^{-1}}(\CN(\af,A)) $.

Then, we have that
$$\hat{\varphi}_{\af^{-1}}(1_{\CN(\af,A)})=1_{\CN(\af,A)} \circ \varphi_\af =1_{\varphi_{\af^{-1}}(\CN(\af,A))}=1_{\CN(\emptyset,A)}.$$
Since $\hat{\varphi}_{\af^{-1}}$ and $\hat{\varphi}_{\af}$ are inverse of each other, we also have $\hat{\varphi}_{\af}(1_{\CN(\emptyset,A)})=1_{\CN(\af,A)}$.

(ii) Let $\af=\af_1 \cdots \af_n$ and $A \in \CI_\af$ such that   $A \subseteq \theta_{\af_2 \cdots \af_n}(B)$ for some $B \in \CI_{\af_1}$.
We first claim that $$1_{\CN(\af_1, B)}\dt_{\af_1}1_{\CN(\af_2 \cdots \af_n, A)}\dt_{\af_2 \cdots \af_n}=1_{\CN(\af,A)}\dt_\af.$$
The left-hand side is
\begin{align*} 1_{\CN(\af_1, B)}\dt_{\af_1}1_{\CN(\af_2 \cdots \af_n, A)}\dt_{\af_2 \cdots \af_n} &=\hat{\varphi}_{\af_1}(\hat{\varphi}_{\af_1^{-1}} (1_{\CN(\af_1, B)})1_{\CN(\af_2 \cdots \af_n, A)}) \dt_{\af} \\
&= \hat{\varphi}_{\af_1}(1_{\CN(\emptyset, B)}1_{\CN(\af_2 \cdots \af_n, A)}) \dt_{\af},
\end{align*}
where the last equality follows from (i). 
Here, for a $\mu :=e^{\af_1}_{\xi_1} \cdots e^{\af_n}_{\xi_n}\mu' \in U_\af$, 
\begin{align*} & \hat{\varphi}_{\af_1}(1_{\CN(\emptyset, B)}1_{\CN(\af_2 \cdots \af_n, A)}) (\mu) \\
 & = 1_{\CN(\emptyset, B)}1_{\CN(\af_2 \cdots \af_n, A)}(\varphi_{\af_1^{-1}}(\mu)) \\
&=1_{\CN(\emptyset, B)}1_{\CN(\af_2 \cdots \af_n, A)} (e^{\af_2}_{\xi_2} \cdots e^{\af_n}_{\xi_n}\mu') \\
&= \left\{
   \begin{array}{ll}
    1  & \hbox{if\ }  B \in f_{\emptyset[\af_2]}(\xi_2)  ~\text{and}~ A \in \xi_{n},  \\
          0 & \hbox{otherwise.}
   \end{array}
\right.
\end{align*}
On the other hand, for a $\mu =e^{\af_1}_{\xi_1} \cdots e^{\af_n}_{\xi_n}\mu' \in U_\af$,
\begin{align*} 1_{\CN(\af,A)}(\mu) &=\left\{
   \begin{array}{ll}
    1  & \hbox{if\ }   A \in \xi_{n},  \\
          0 & \hbox{otherwise.}
   \end{array}
\right.
\end{align*}
Since $ \xi_n \ni A \subseteq \theta_{\af_2 \cdots \af_n}(B) (\in \CI_{\af_2 \cdots \af_n} \subset \CI_{\af_n})$, we have  $\theta_{\af_2 \cdots \af_n}(B) \in \xi_n \cap \CI_{\af_2 \cdots \af_n}$, and hence,  we see that $ B \in f_{\emptyset[\af_2]}(\xi_2)=r(e^{\af_2}_{\xi_2} \cdots e^{\af_n}_{\xi_n})=f_{\emptyset[\af_2 \cdots \af_n]}(g_{(\af_2 \cdots \af_{n-1})\af_n}(\xi_n))$ by Lemma \ref{range of path},
 Thus, we conclude that
$$\hat{\varphi}_{\af_1}(1_{\CN(\emptyset, B)}1_{\CN(\af_2 \cdots \af_n, A)})=1_{\CN(\af,A)}.$$
Hence, it follows that
 $$1_{\CN(\af_1, B)}\dt_{\af_1}1_{\CN(\af_2 \cdots \af_n, A)}\dt_{\af_2 \cdots \af_n}=1_{\CN(\af,A)}\dt_\af.$$

Now, for $\af=\af_1 \cdots \af_n \in \CW^{\geq 1}$ and $A \in \CI_\af$ such that   $A \subseteq \theta_{\af_2 \cdots \af_n}(B)$ for some $B \in \CI_{\af_1}$, applying the preceding argument pairwise from right to left yields 
\begin{align*} & (1_{\CN(\af_1, B)}\dt_{\af_1})( 1_{\CN(\af_2, \theta_{\af_2}(B))}\dt_{\af_2}) \cdots (1_{\CN(\af_{n-1}, \theta_{\af_2 \cdots \af_{n-1}}(B))}\dt_{\af_{n-1}}) (1_{\CN(\af_n, A)}\dt_{\af_n})  \\
&=(1_{\CN(\af_1, B)}\dt_{\af_1})( 1_{\CN(\af_2, \theta_{\af_2}(B))}\dt_{\af_2}) \cdots (1_{\CN(\af_{n-1}\af_n, A)})\dt_{\af_{n-1}\af_n} \\
&~~~\vdots\\
&=(1_{\CN(\af_1, B)}\dt_{\af_1}) (1_{\CN(\af_2 \cdots \af_n, A)}\dt_{\af_2 \cdots \af_n})\\
&= 1_{\CN(\af,A)}\dt_\af,
\end{align*}
completing the proof of (ii).

(iii)  We have  $(1_{\CN(\af,A)}\dt_\af)^*=\hat{\varphi}_{\af^{-1}}(1_{\CN(\af,A)})\dt_{\af^{-1}}=1_{\CN(\emptyset,A)}\dt_{\af^{-1}}$ by (i).

(iv) Applying (i), we have that 
\begin{align*}  (1_{\CN(\af,A)}\dt_\af)(1_{\CN(\af,A)}\dt_\af)^* &=(1_{\CN(\af,A)}\dt_\af)(1_{\CN(\emptyset,A)}\dt_{\af^{-1}}) \\
&=\hat{\varphi}_\af(\hat{\varphi}_{\af^{-1}}(1_{\CN(\af,A)})1_{\CN(\emptyset,A)})\dt_\emptyset \\
&=\hat{\varphi}_\af(1_{\CN(\emptyset,A)}1_{\CN(\emptyset,A)})\dt_\emptyset \\
&=\hat{\varphi}_\af(1_{\CN(\emptyset,A)})\dt_\emptyset \\
&=1_{\CN(\af,A)}\dt_\emptyset.
\end{align*}

(v) Applying (i) again, we see that  \begin{align*} (1_{\CN(\af,A)}\dt_\af)(1_{\CN(\emptyset, A)}\dt_{\bt^{-1}}) &= \hat{\varphi}_\af(\hat{\varphi}_{\af^{-1}}(1_{\CN(\af,A)})1_{\CN(\emptyset, A)})\dt_{\af\bt^{-1}} \\
&=\hat{\varphi}_\af(1_{\CN(\emptyset,A)}1_{\CN(\emptyset, A)})\dt_{\af\bt^{-1}} \\
&=\hat{\varphi}_\af(1_{\CN(\emptyset,A)})\dt_{\af\bt^{-1}} \\
&=1_{\CN(\af,A)}\dt_{\af\bt^{-1}}. 
\end{align*}

(vi) Let $A \in \CB$, $\af \in \CW^{\geq 1}$ and $B \in \CI_\af$.
Note first that $$(1_{\CN(\emptyset,A)}\dt_\emptyset)(1_{\CN(\af,B)}\dt_\af)=(1_{\CN(\emptyset,A)}1_{\CN(\af,B)}) \dt_\af.$$
We  claim that $$1_{\CN(\emptyset,A)}1_{\CN(\af,B)}=\hat{\varphi}_\af(1_{\CN(\emptyset,B \cap \theta_\af(A))} ).$$
For a  $\mu=(e^{\mu_i}_{\xi_i}) \in \partial E$, we have that 
\begin{align*} 1_{\CN(\emptyset,A)}1_{\CN(\af,B)}(\mu) &=\left\{
   \begin{array}{ll}
    1  & \hbox{if\ }  \mu_1=\af,  ~A \in f_{\emptyset[\af]}(\xi_1) ~\text{and}~ B \in \xi_1,  \\
          0 & \hbox{otherwise}
   \end{array}
\right. \\
 &=\left\{
   \begin{array}{ll}
    1  & \hbox{if\ }  \mu_1=\af, ~ \theta_\af(A) \cap B \in \xi_1, \\
          0 & \hbox{otherwise,}
   \end{array}
\right.
\end{align*}
 and that 
 \begin{align*} \hat{\varphi}_\af(1_{\CN(\emptyset,B \cap \theta_\af(A))} )(\mu)&= 1_{\CN(\emptyset,B \cap \theta_\af(A))}(\varphi_{\af^{-1}}(\mu))\\ &=\left\{
   \begin{array}{ll}
    1  & \hbox{if\ }  B \cap \theta_\af(A) \in f_{\emptyset[\mu_2]}(\xi_2), \\
          0 & \hbox{otherwise.}
   \end{array}
\right.
\end{align*}
Here,  since 
$ h_{[\af]\emptyset}(\xi_1) \cap \CI_\af =g_{(\af)\emptyset}(h_{[\af]\emptyset}(\xi_1) )=\xi_1$, one can  see that $B \cap \theta_\af(A) \in f_{\emptyset[\mu_2]}(\xi_2)(=h_{[\af]\emptyset}(\xi_1))$ if and only if $\theta_\af(A) \cap B \in \xi_1$. Thus, we have  $1_{\CN(\emptyset,A)}1_{\CN(\af,B)}= \hat{\varphi}_\af(1_{\CN(\emptyset,B \cap \theta_\af(A))} ).$

Then it follows that
\begin{align*} (1_{\CN(\af,B)}\dt_\af)(1_{\CN(\emptyset,\theta_\af(A))}\dt_\emptyset)&=\hat{\varphi}_\af(\hat{\varphi}_{\af^{-1}}(1_{\CN(\af,B)}) 1_{\CN(\emptyset,\theta_\af(A))})\dt_\af \\
&=\hat{\varphi}_\af(1_{\CN(\emptyset,B)} 1_{\CN(\emptyset,\theta_\af(A))})\dt_\af \\
&=\hat{\varphi}_\af(1_{\CN(\emptyset,B \cap \theta_\af(A))} )\dt_\af \\
&= (1_{\CN(\emptyset,A)}1_{\CN(\af,B)})\dt_\af \\
&=(1_{\CN(\emptyset,A)}\dt_\emptyset)(1_{\CN(\af,B)}\dt_\af).  
\end{align*}
\end{proof}

\begin{lem}\label{app. unit}  For $\af, \bt \in \CW^{*}$,  $\{1_{\CN(\af,A)}\}_{A\in\CI_\bt}$ is an approximate identity for $C_0(U_{\af\bt^{-1}})$.
\end{lem}

\begin{proof}Note that $\CI_\bt$ is a directed set with the order given by inclusion. To show that $\{1_{\CN(\af,A)}\}_{A \in \CI_\bt}$ is an approximate identity for   $C_0(U_{\af\bt^{-1}})$, 
let $g \in C_0(U_{\af\bt^{-1}})$ and $\epsilon > 0$ be given. 
Take $h \in C_c(U_{\af\bt^{-1}})$ such that $\|g-h\| < \epsilon$. Put $K:= \operatorname{supp}(h)$. Since $K$ is compact,  there exist $\CN(\af,A_1), \cdots, \CN(\af,A_n)$ such that $K \subseteq \bigcup_{i=1}^n \CN(\af,A_i) = \CN(\af, \bigcup_{i=1}^n A_i)$, where the last equality follows from Lemma \ref{union of open sets}.    Take $A_0:=\bigcup_{i=1}^n A_i \in \CI_\bt$. Then if $A \geq A_0$, then 
$\CN(\af, A_0) \subset \CN(\af,A)$,  and hence, we have 
\begin{align*} \|g-g1_{\CN(\af,A)}\| & \leq \| g-h \| +\| h -h1_{\CN(\af,A)}  \|+\| h1_{\CN(\af,A)} -g1_{\CN(\af,A)}  \| \\
&   \leq \epsilon + \|h-g\|\|1_{\CN(\af,A)}\| \\
& \leq 2\epsilon.
\end{align*}
 So, we are done.
\end{proof}

\begin{prop}\label{generator of the p.c.p} Let $C^*(\{1_{\CN(\emptyset,A)}\dt_\emptyset, 1_{\CN(\af,B)}\dt_\af\}) \subseteq C_0(\partial E) \rtimes_{\hat{\varphi}} \mathbb{F} $ denote the $C^*$-subalgebra generated by $\{1_{\CN(\emptyset,A)}\dt_\emptyset, 1_{\CN(\af,B)}\dt_\af: A \in \CB ~\text{and}~\af \in \CL, B \in \CI_\af \}$. Then 
$$C_0(\partial E) \rtimes_{\hat{\varphi}} \mathbb{F} = C^*(\{1_{\CN(\emptyset,A)}\dt_\emptyset, 1_{\CN(\af,B)}\dt_\af\}).$$
\end{prop}

\begin{proof} Let $C_0(\partial E)\dt_\emptyset$ denote the canonical image of $C_0(\partial E)$ in $C_0(\partial E) \rtimes_{\hat{\varphi}} \mathbb{F}$. We first claim that
$C_0(\partial E)\dt_\emptyset \subset C^*(\{1_{\CN(\emptyset,A)}\dt_\emptyset, 1_{\CN(\af,B)}\dt_\af\})$. To prove this, it is enough to show that $1_{\CN(\af,A)}\dt_\emptyset \in C^*(\{1_{\CN(\emptyset,A)}\dt_\emptyset, 1_{\CN(\af,B)}\dt_\af\})$ for $\af \in \CW^{\geq 1}$ and $A \in \CI_\af$ by Lemma \ref{C_0}. 
Let $\af=\af_1 \cdots \af_n$ and $B \in \CI_{\af_1}$ such that $A \subseteq \theta_{\af_2 \cdots \af_n}(B)$.
Then, by Lemma \ref{computations}(ii),(iv) we have 
\begin{align*} 1_{\CN(\af,A)}\dt_\emptyset &= (1_{\CN(\af,A)}\dt_\af)(1_{\CN(\af,A)}\dt_\af)^* \\
&=(1_{\CN(\af_1, B)}\dt_{\af_1} 1_{\CN(\af_2, \theta_{\af_2}(B))}\dt_{\af_2} \cdots 1_{\CN(\af_n, A)}\dt_{\af_n}) \\
& \hskip 3pc (1_{\CN(\af_1, B)}\dt_{\af_1} 1_{\CN(\af_2, \theta_{\af_2}(B))}\dt_{\af_2} \cdots 1_{\CN(\af_n, A)}\dt_{\af_n})^*,
\end{align*}
completing the proof.

Next we show that $C_0(\partial E) \rtimes_{\hat{\varphi}} \mathbb{F} = C^*(\{1_{\CN(\emptyset,A)}\dt_\emptyset, 1_{\CN(\af,B)}\dt_\af\})$. It is clear that 
$C^*(\{1_{\CN(\emptyset,A)}\dt_\emptyset, 1_{\CN(\af,B)}\dt_\af\}) \subseteq C_0(\partial E) \rtimes_{\hat{\varphi}} \mathbb{F}$. To show the reverse inclusion, 
let $f_t\dt_t \in C_0(\partial E) \rtimes_{\hat{\varphi}} \mathbb{F}$ with $f_t \in  D_t= C_0(U_t)$ and  $t \in \mathbb{F}$. If $t=\emptyset$, then $f_\emptyset \dt_\emptyset \in C^*(\{1_{\CN(\emptyset,A)}\dt_\emptyset, 1_{\CN(\af,B)}\dt_\af\})$ by Lemma \ref{C_0}. 
If $t \neq \emptyset$, we may assume that $t=\af\bt^{-1}$ with $\af,\bt \in \CW^*$.
Since $\{1_{\CN(\af,A)}\}_{A\in\CI_\bt}$ is an approximate identity for $C_0(U_{\af\bt^{-1}})$ by Lemma \ref{app. unit}, we have $ \lim_{A \to\infty} g1_{\CN(\af,A)}=g$ for $g \in C_0(U_{\af\bt^{-1}})$. Thus, 
  we have 
\begin{align*} f_{\af\bt^{-1}}\dt_{\af\bt^{-1}} &= \big( \lim_{A \to\infty} f_{\af\bt^{-1}}1_{\CN(\af,A)}\big) \dt_{\af\bt^{-1}}  \\
&=  \lim_{A \to\infty} \big( f_{\af\bt^{-1}}1_{\CN(\af,A)} \dt_{\af\bt^{-1}} \big) \\
&= \lim_{A \to\infty} \big( f_{\af\bt^{-1}}\dt_\emptyset \big)\big( 1_{\CN(\af,A)} \dt_{\af\bt^{-1}} \big) \\
&= \lim_{A \to\infty} \big( f_{\af\bt^{-1}}\dt_\emptyset \big) \big( 1_{\CN(\af,A)} \dt_{\af} \big) \big( 1_{\CN(\emptyset,A)} \dt_{\bt^{-1}} \big)\\
&=  \lim_{A \to\infty} \big( f_{\af\bt^{-1}}\dt_\emptyset \big) \big( 1_{\CN(\af,A)} \dt_{\af} \big) \big( 1_{\CN(\beta,A)} \dt_{\bt} \big)^*,
\end{align*}
where the fourth equality follows from Lemma \ref{computations}(v), and the fifth equality follows from Lemma \ref{computations}(iii).
We then can conclude, using the case $t=\emptyset$ above, that $ f_{\af\bt^{-1}}\dt_{\af\bt^{-1}} \in C^*(\{1_{\CN(\emptyset,A)}\dt_\emptyset, 1_{\CN(\af,B)}\dt_\af\}) $. So, we are done. 
\end{proof}

\section{Partial actions and generalized Boolean dynamical system groupoids are isomorphic}\label{sec: groupoids}

Let  $(\CB,\CL,\theta, \CI_\af)$ be a generalized Boolean dynamical system,  $E=(E^1,d,r)$ be the associated  topological correspondence and  $\Phi=(\{U_t\}_{t \in \mathbb{F}}, \{\varphi_t\}_{t \in \mathbb{F}})$ be the partial action of  $\mathbb{F}$ on $\partial E$. 
In this section, we show that the groupoid associated with the partial action  is isomorphic to the boundary path groupoid $\Gamma(\partial E, \sm_E)$ associated with the topological correspondence as studied in \cite[Sect. 7]{CasK1}.   As a result, we have that 
 the partial crossed product $C^*$-algebras obtained from the partial action  is isomorphic to the $C^*$-algebra of the generalized Boolean dynamical system.

 We first recall the boundary path groupoid $\Gamma(\partial E, \sm_E)$. To easy notation, put $\sm:=\sm_E$.

\begin{dfn}\label{boundary path groupoid 2} Let  $(\CB,\CL,\theta, \CI_\af)$ be a generalized Boolean dynamical system and $E=(E^1,d,r)$ be the associated  topological correspondence.  We define the {\it boundary path groupoid}  $\Gamma(\partial E, \sigma_E)$ to be the Renault-Deaconu groupoid, that  is, 
\begin{align*} \Gamma &:= \Gamma(\partial E, \sigma_E) \\
& =\{(\mu, k-l, \nu) \in \partial E \times \Z \times \partial E : \mu \in \text{dom}(\sigma^k), \nu \in \text{dom}(\sigma^l),  
 \sigma^k( \mu)=\sigma^l(\nu)\}.
\end{align*}
The unit space is defined by $\Gamma^{(0)}:=\{(\mu, 0 , \mu): \mu \in \partial E\}$. For $(\mu, n, \nu), (\nu, m, \dt) \in \Gamma(\partial E, \sigma_E) $, we define the multiplication and the inverse by 
\begin{align*} (\mu, n, \nu)(\nu, m, \dt)&:=(\mu, n + m,\dt), \\
(\mu, n, \nu)^{-1}&:=( \nu , -n,\mu),
\end{align*}
and the range map and  the source map by
$$r_\Gamma(\mu, n, \nu) := (\mu, 0, \mu), \hskip 1pc  s_\Gamma(\mu, n, \nu):= (\nu, 0, \nu).$$
Define the topology on $ \Gamma(\partial E, \sigma_E) $ to be generated by the basic open set 
\begin{align}\CU(U,k_1,k_2,V):=\{(\mu, k_1-k_2, \nu):\mu \in U, \nu \in V, \sigma^{k_1}(\mu)= \sigma^{k_2}(\nu)\},\end{align}
where $U \subseteq \text{dom}(\sm^{k_1})$, $V \subseteq \text{dom}(\sm^{k_2})$ are open in $\partial E$, $\sm^{k_1}$ is injective on $U$ and $\sm^{k_2}$ is injective on $V$. 
\end{dfn}

We note that  $\Gamma(\partial E, \sigma_E)$ is a locally compact  Hausdorff  ample groupoid and that the unit space $\Gamma^{(0)}$ is identified with $\partial E$.

Now consider  the groupoid 
$$\CG:=\mathbb{F} \ltimes_\varphi \partial E =\{(\mu,t,\nu) \in \partial E \times \mathbb{F} \times \partial E: \nu \in U_{t^{-1}} ~\text{and}~ \mu=\varphi_t(\nu)\}$$ associated to the  partial action  $\Phi=(\{U_t\}_{t \in \mathbb{F}}, \{\varphi_t\}_{t \in \mathbb{F}})$ of  $\mathbb{F}$ on $\partial E$ as in Section \ref{section partial actions}. The unit space $\CG^0$ is also identified with $\partial E$. 

The following characterization of elements in $\CG$ is used throughout this section without reference. 
\begin{lem} We have that $(\mu, t, \nu) \in \CG$ if and only if $(\mu, t, \nu)  $ is such that  $\CP(\mu)_{1,|\af|}=\af$ and $\CP(\nu)_{1,|\bt|}=\bt$ 
for some $\af,\bt \in \CW^*$, $t=\af\bt^{-1}$ and $\varphi_{\af^{-1}}(\mu)=\varphi_{\bt^{-1}}(\nu)$. 
\end{lem}

\begin{proof} If  $(\mu, t, \nu) \in \CG$, then $\nu \in U_{t^{-1}}$, and hence, $U_t \neq \emptyset$. Thus, $t=\af\bt^{-1}$ for some $\af, \bt \in \CW^*$. Since $U_{t^{-1}}=U_{\bt\af^{-1}} \subset U_\bt$, we have $\CP(\nu)_{1,|\bt|}=\bt$. Also, since $\mu=\varphi_{\af\bt^{-1}}(\nu)$, we have
$\CP(\mu)_{1,|\af|}=\af$ and  $\varphi_{\af^{-1}}(\mu) =\varphi_{\af^{-1}}(\varphi_{\af\bt^{-1}}(\nu))=\varphi_{\bt^{-1}}(\nu) $.

Let $(\mu, t, \nu)  $ be such that  $\CP(\mu)_{1,|\af|}=\af$ and $\CP(\nu)_{1,|\bt|}=\bt$ 
for some $\af,\bt \in \CW^*$, $t=\af\bt^{-1}$ and $\varphi_{\af^{-1}}(\mu)=\varphi_{\bt^{-1}}(\nu)$.
Then, $\mu=\varphi_\af(\varphi_{\bt^{-1}}(\nu))=\varphi_{\af\bt^{-1}}(\nu)$ and $\nu \in U_{\bt\af^{-1}}$. So,  $(\mu, t, \nu) \in \CG$.
\end{proof}

\begin{lem} The groupoid $\CG=\mathbb{F} \ltimes_\varphi \partial E$ is an ample groupoid.
\end{lem}

\begin{proof} We first show that $\CG$ is $\acute{e}$tale. Let $(\mu, \af\bt^{-1}, \nu) \in \CG$. Let $W=\varphi_{\af^{-1}}(U_\af) \cap \varphi_{\bt^{-1}}(U_{\bt})$, and put $W_\af=\varphi_\af(W)$ and $W_\bt=\varphi_\bt(W)$.  Then $$U:=(W_\af \times \{\af\bt^{-1}\} \times W_\bt) \cap \CG$$ is an open neighborhood of $(\mu, \af\bt^{-1}, \nu) (=(\varphi_{\af\bt^{-1}}(\nu), \af\bt^{-1}, \varphi_{\bt\af^{-1}}(\mu)))$ such that $r_\CG|_U$ is just the projection onto the first coordinate, and is thus a  homeomorphism. Thus, $\CG$ is $\acute{e}$tale. Since then $\mathbb{F}$ is discrete and $\partial E$ has a basis of compact-open sets by Remark \ref{basis of partial E}, it follows that $\CG$ has a basis of compact open sets. So, $\CG$ is an ample groupoid. 
\end{proof}

\begin{thm}\label{iso groupoids} The map $\Theta: \mathbb{F} \ltimes_\varphi \partial E \to \Gamma(\partial E, \sm_E)$ defined by 
$$\Theta((\mu,\af\bt^{-1}, \nu))=(\mu, |\af|-|\bt|, \nu)$$ is a groupoid isomorphism. 
\end{thm}

\begin{proof} We first show that $\Theta$ is well-defined. For $t=(\af\dt)(\bt\dt)^{-1}$ such that $t=\af\bt^{-1}$ is in reduced form. Then $|\af\dt|-|\bt\dt|=|\af|-|\bt|$. Since $\Theta$ only changes the second coordinate, it follows that $\Theta$ is well-defined.

Clearly, $\Theta$ is onto. 
To show that $\Theta$ is injective, let $\Theta((\mu,\af\bt^{-1},\nu))=\Theta((\gm, st^{-1}, \dt))$.  Then $\mu=\gm$ and $\nu=\dt$, and hence,  $\CP(\mu)=\CP(\gm)$ and $\CP(\nu)=\CP(\dt)$.
Say $\CP(\mu)=\af\zeta$, $\CP(\nu)=\bt\zeta$, $\CP(\gm)=s\eta$ and $\CP(\dt)=t\eta$. Then, 
$\af\zeta=s\eta$ and $\bt\zeta=t\eta$. We may assume that $\af=s\af'$ (if not, then $s=\af s'$ and the same arguments holds). Then, $\af\zeta=s\af'\zeta=s\eta$, implying that $\eta=\af'\zeta$. Then, $\bt\zeta=t\af'\zeta$, which implies that $\bt=t\af'$. Then we compute in $\mathbb{F}$:
$$\af\bt^{-1}=s\af'(t\af')^{-1}=st^{-1}.$$
Thus, $\Theta$ is injective. 

We next note that it is clear that $(\Theta \times \Theta) (\CG^{(2)}) \subseteq \Gamma^{(2)}$. Now, to show that $\Theta$ preserves multiplication, let $(\mu,\af\bt^{-1}, \nu), (\nu, st^{-1}, \dt) \in \CG$. If $\bt=s\bt'$ (if not, then $s=\bt s'$ and the same arguments holds), then 
\begin{align*} \Theta((\mu,\af\bt^{-1}, \nu)(\nu, st^{-1}, \dt)) &= \Theta((\mu, \af(s\bt')^{-1}st^{-1}, \dt)) \\
&=\Theta((\mu, \af (t\bt')^{-1}, \dt)) \\
&=(\mu, |\af|-|t\bt'|, \dt),
\end{align*}
and 
\begin{align*}  \Theta((\mu,\af\bt^{-1}, \nu))\Theta((\nu, st^{-1}, \dt)) &= (\mu, |\af|-|\bt|, \nu)(\nu, |s|-|t|, \dt) \\
&= (\mu,|\af|-|s|-|\bt'|+  |s|-|t|  , \dt) \\
&=(\mu, |\af|-|t\bt'|,\dt).
\end{align*}
Thus,  $ \Theta((\mu,\af\bt^{-1}, \nu)(\nu, st^{-1}, \dt))=\Theta((\mu,\af\bt^{-1}, \nu))\Theta((\nu, st^{-1}, \dt)) $. 

Lastly, to show that $\Theta$ is a homeomorphism, since both $\Gamma(\partial E, \sm_E)$ and $\CG$ are ample groupoids, it suffices to show that their unit spaces are homeomorphic. But, this follows from the fact that both $\CG^{(0)}$ and $\Gamma^{(0)}$ are homeomorphic to $\partial E$. So, we are done.
\end{proof}

\begin{cor}\label{cor: iso} Let  $(\CB,\CL,\theta, \CI_\af)$ be a generalized Boolean dynamical system. Then we have  $$C^*(\CB,\CL,\theta,\CI_\af) \cong C_0(\partial E) \rtimes_{\hat{\varphi}} \mathbb{F}.$$ 
\end{cor}

\begin{proof} By \cite{A}, we have $C_0(\partial E) \rtimes_{\hat{\varphi}} \mathbb{F} \cong C^*(\mathbb{F} \ltimes_\varphi \partial E)$. Hence, we have that
$$C^*(\CB,\CL,\theta,\CI_\af) \cong C^*(\Gamma(\partial E, \sm_E)) \cong C^*(\mathbb{F} \ltimes_\varphi \partial E) \cong C_0(\partial E) \rtimes_{\hat{\varphi}} \mathbb{F},$$  
where the first isomorphism follows from \cite[Corollary 7.11]{CasK1}.
\end{proof}

\begin{remark} By \cite[Theorem 4.3]{ExelLaca2003}, if $N: \CL \to (0, \infty)$ is any function, then there exists a unique strongly continuous one-parameter group $\sm$ of automorphism of 
$C_0(\partial E) \rtimes_{\hat{\varphi}} \mathbb{F}$ such that 
\[\sm_t(f\dt_\af)=N(\af)^{it}f \dt_\af ~\text{and}~ \sm_t(g\dt_\emptyset)=g \dt_\emptyset \]
for all $t \in \mathbb{R}$, $\af \in \CL$, $f \in D_\af$ and $g \in D_\emptyset$.
If we let $N(\af)=\exp(1)$ for every $\af \in \CL$, then we obtain a strongly continuous action (using the same notation) $\sm: \mathbb{T} \to \operatorname{Aut}(C_0(\partial E) \rtimes_{\hat{\varphi}} \mathbb{F})$ such that 
\[\sm_z(f\dt_\af)=zf \dt_\af ~\text{and}~ \sm_z(g\dt_\emptyset)=g \dt_\emptyset \]
for all $t \in \mathbb{R}$, $\af \in \CL$, $f \in D_\af$ and $g \in D_\emptyset$.

Then, one also can directly show that there is a $*$-isomorphism $\psi$ from $C^*(\CB,\CL,\theta,\CI_\af)$ onto $C_0(\partial E) \rtimes_{\hat{\varphi}} \mathbb{F}$ such that 
$$\psi(p_A)=1_{\CN(\emptyset,A)}\dt_\emptyset  ~\text{and}~ \psi(s_{\af,B})=1_{\CN(\af,B)}\dt_\af$$
for  $A \in \CB$, $\af \in \CL$ and $B \in \CI_\af$ by showing that 
 $$\{1_{\CN(\emptyset,A)}\dt_\emptyset, 1_{\CN(\af,B)}\dt_\af: A \in \CB, \af \in \CL ~\text{and}~B \in \CI_\af \}$$ is a $(\CB, \CL,\theta,\CI_\af)$-representation in  $C_0(\partial E) \rtimes_{\hat{\varphi}} \mathbb{F}$ and using the gauge-invariant uniqueness theorem.
\end{remark}

\section{Partial actions of the free group on Stone spaces}\label{section model}

In this section, we show how we can model certain partial actions of the free group using Boolean dynamical systems. We then consider the C*-algebraic version of this result by looking at partial actions of the free group on commutative C*-algebras generated by projections.

\begin{thm}\label{partial action model}(cf. \cite[Theorem 7.10]{BCGW}) Let $X$ be a Stone space, that is, a dual of a Boolean algebra, and let 
$\rho=(\{V_t\}_{t \in \mathbb{F}}, \{\rho_t\}_{t \in \mathbb{F}})$ be a semi-saturated, orthogonal topological partial action of a free group $\mathbb{F}$ on $X$ such that $V_\af$ is clopen for all generators $\af$ of $\mathbb{F}$. Then there exist a  Boolean dynamical system $(\CB,\CL,\theta)$ and a homeomorphism $f: X \to \partial E$, such that $f$ is equivariant with respect to the actions $\rho$ and $\varphi$, where $\partial E$ is the boundary path space associated with $(\CB,\CL,\theta, \CR_\af)$ and $\varphi$ is the partial action given in Section 3. In particular, $\mathbb{F} \ltimes_\rho X$ and  $\mathbb{F} \ltimes_\varphi \partial E $ are isomorphic as topological groupoids.
\end{thm}

\begin{proof} Let $\CL$ be the set of generators of $\mathbb{F}$ and $\CB$ the set of compact open subsets of $X$. For each $\af \in \CL$, define $\theta_\af: \CB \to \CB$ by $\theta_\af(A)=\rho_{\af^{-1}}(A \cap  V_\af)$.
Then, we clearly have $\theta_\af(\emptyset)=\rho_{\af^{-1}}(\emptyset)=\emptyset$ and 
 $\theta_\af(A\cup B)= \theta_\af(A) \cup \theta_\af(B)$. Since $\rho_{\af^{-1}}$ is bijective for each $\af \in \CL$, we have 
\begin{align*} \theta_\af(A\cap B) &=\rho_{\af^{-1}}( (A \cap B) \cap  V_\af) \\
&=\rho_{\af^{-1}}( (A \cap V_\af) \cap (B \cap V_\af)) \\
&=\rho_{\af^{-1}}( A \cap V_\af)  \cap \rho_{\af^{-1}} (B \cap V_\af) \\
&=\theta_\af(A) \cap \theta_\af(B).
\end{align*}
Also, since we have 
$\theta_\af(A \setminus B) \cap \theta_\af(B)=\theta_\af(\emptyset)=\emptyset$ and 
$$\theta_\af(A \setminus B) \cup \theta_\af(B) =\theta_\af(A\cup B)=\theta_\af(A)\cup \theta_\af(B)= (\theta_\af(A) \setminus \theta_\af(B)) \cup \theta_\af(B),$$
 it follows that $\theta_\af(A \setminus B) =\theta_\af(A) \setminus \theta_\af(B)$.
Thus, $ \theta_\af$ is an action for each $\af \in \CL$, and hence, $(\CB,\CL,\theta)$ is a Boolean dynamical system. 
Note that  $\af\in\Delta_A (=\{\af \in \CL: \theta_\af(A)=\rho_{\af^{-1}}(A \cap V_\af) \neq \emptyset\})$ if and only if $A\cap V_{\af}\neq\emptyset$.

For each $\af \in \CL$, we claim that 
\begin{align*} \CR_\af&=\{A \in \CB: A \subseteq \rho_{\af^{-1}} (B \cap V_\af) ~\text{for some}~ B \in \CB \} \\
&= \{A \in \CB: A \subseteq V_{\af^{-1}} \}. 
\end{align*}
It is clear that $\CR_\af \subset \{A \in \CB: A \subseteq V_{\af^{-1}} \}$. 
For the converse inclusion,  choose $A\in\CB$ such that $A\subseteq V_{\af^{-1}}$ and take $B=\rho_{\af}(A)\subseteq V_{\af}$. Then $A=\rho_{\af^{-1}}(B)\in\CR_\af$.

Notice that we have $X= \big( X \setminus \bigsqcup_{\af \in \CL} V_\af \big) \sqcup \big( \bigsqcup_{\af \in \CL} V_\af \big )$ since the partial action $\rho$ is orthogonal.
We first associate an element of $\CW^{\leq \infty}$  with a given $x \in X$ as follows: 
given $x \in X$,  there is either a unique letter $\af_1 \in \CL$ such that $x \in V_{\af_1}$ or $x \in  X \setminus \bigsqcup_{\af \in \CL} V_\af $. In the first case, the same dichotomy applies to $\rho_{\af_1^{-1}}(x)$, and so either there is a unique letter $\af_2 \in \CL$ such that $\rho_{\af_1^{-1}}(x) \in V_{\af_2}$ or $\rho_{\af_1^{-1}}(x) \in X \setminus \bigsqcup_{\af \in \CL} V_\af$. 
We  note here that $\rho_{\af_2^{-1}}(\rho_{\af_1^{-1}}(x))=\rho_{(\af_1\af_2)^{-1}}(x)$ since the partial action is semi-saturated, and that $x \in V_{\af_1\af_2}$. 
This process might either stop and we get a finite word $\af_1 \cdots \af_n$ (possibly the empty word) or it does not stop and we get an infinite word $\af_1\af_2 \cdots$. We denote by $\af:=\af_x$ this word and observe 
 that $x \in V_{\af_{1,i}}$ for all $1 \leq i \leq |\af_x|$ (understanding that $i \leq |\af_x|$ means $i < \infty$ if $|\af_x|=\infty$).
 Since $\rho_{(\af_{2,i})^{-1}}(V_{\af_1^{-1}} \cap V_{\af_2}) \neq \emptyset$,  we have $\CR_{\af_{1,i}} \neq \{\emptyset\}$ for all $1 \leq i \leq |\af_x|$, and hence, we have a path $\af_x \in \CW^{\leq \infty}$.

 Now,  if $|\af_x|=0$, we define $\eta=\{A\in\CB:x\in A\}$. If  $|\af_x|\geq 1$, we  define $\eta_i=\{ A \in \CR_{\af_i}: \rho_{\af_{1,i}}^{-1}(x) \in A\}$  for each $1 \leq i \leq |\af_x|$.  Then, $\eta$ and $\eta_i$ are  ultrafilters for all $1 \leq i \leq |\af_x|$ and we have 
\begin{align*}  & f_{\emptyset[\af_{i+1}]}(\eta_{i+1}) \\ &= \{ A \in \CB: \theta_{\af_{i+1}}(A) \in \eta_{i+1}\} \\
&=\{  A \in \CB: \rho_{\af_{1,i+1}^{-1}}(x) \in \rho_{\af_{i+1}^{-1}}(A \cap V_{\af_{i+1}})  \} \\
&=\{ A \in \CB : \rho_{\af_{1,i}^{-1}}(x) \in A \cap V_{\af_{i+1}}  \} \\
&=h_{[\af_i]\emptyset}(\eta_i)
\end{align*}
for  $1 \leq i < |\af_x|$ . 
We then define a map $f: X \to \partial E$ by 
\[f(x)=\begin{cases}
\eta & \text{if } |\af_x|=0, \\
(e^{\af_i}_{\eta_i})_{1 \leq i \leq |\af_x|} &\text{if }|\af_x|\geq 1.
\end{cases}\]
Then $f$ is well-defined and  
it is easy to see that the map $f$ is injective.

To show that $f$ is surjective, let $\mu=(e^{\af_i}_{\xi_i}) \in \partial E$. We first note that we cannot have $|\CP(\mu)| \geq 1$ and $f_{\emptyset[\af_1]}(\xi_1)=\emptyset$, because if $A \in \xi_1$, then 
$\theta_{\af_1}(\rho_{\af_1}(A))=A$ so that $\rho_{\af_1}(A) \in f_{\emptyset[\af_1]}(\xi_1)$. 
So, $r(\mu) \neq \emptyset$ for any given $\mu$. Now, let $x_\mu \in X$ be the element corresponding to $r(\mu) $ via Stone duality. 
We then claim that $f(x_\mu)=\mu$. 
Suppose first that $\CP(\mu)=\emptyset$. So, $\mu \in E_{sg}^0$.
If $x_\mu \in \bigsqcup_{\af \in \CL} V_\af$,  then there exits $\af \in \CL$ such that $x_\mu \in V_\af$. Choose $A \in \CB $ such that $x_\mu \in A$. Then,  $x_\mu \in A \cap V_\af (\in \CB)$, and hence,  $A \cap V_\af \in \mu$. On the other hand, since $\Delta_{A \cap V_\af}=\{\af\}$ and  $ \big( A \cap V_\af \big) \cap \big( X \setminus \bigsqcup_{\af \in \CL} V_\af \big)=\emptyset$,  we have $A \cap V_\af \in \CB_{\reg}$, which contradicts to  \cite[Lemma 7.9]{CasK1}.
Hence, in this case, $\af_{x_\mu}=\emptyset$ and $f(x_\mu)=\mu$. 
Suppose next that $|\CP(\mu)| \geq 1$.   We prove that,  for $1 \leq i \leq |\CP(\mu)|$, $x_\mu \in V_{\af_{1,i}}$ and $\xi_i$ is the ultrafilter in $\CR_{\af_{i}}$ corresponding to the point $\rho_{\af_{1,i}^{-1}}(x_\mu)$. 
For $i=1$, we must have $x_\mu \in V_{\af_1}$ because otherwise we would find $A \in f_{\emptyset[\af_1]}(\xi_1)$ such that $\theta_{\af_1}(A)=\emptyset$ which does not exist. 
Next, suppose that $\xi_1$ corresponds to a point $y \in X$  different from $\rho_{\af_{1}^{-1}}(x_\mu) $. We then have an element $B \in \xi_1$  such that 
$y \in B$  but $\rho_{\af_{1}^{-1}}(x_\mu) \notin B$. 
Then $A:=\rho_{\af_1}(B) \in \CB$ is such that $\theta_{\af_1}(A)=B$ so that $A \in f_{\emptyset[\af_1]}(\xi_1)$, but $x_\mu \notin A$, which is a contradiction. 
Using induction and repeating the same argument we conclude that $\xi_i$ is the ultrafilter in $\CR_{\af_i}$ corresponding to $\rho_{\af_{1,i}^{-1}}(x_\mu)$.
Now, if $|\af|$ is finite, as for the case of the empty word, we have that $\rho_{\af^{-1}}(x_\mu) \in   X \setminus \bigsqcup_{\af \in \CL} V_\af $, and hence, $\af_{x_\mu}=\af$ and $f(x_\mu)=\mu$. On the other hand, if $|\af|=\infty$, then $\rho_{\af_{1,i}^{-1}}(x) \notin X \setminus \bigsqcup_{\af \in \CL} V_\af $ for all $i$ and again we have  $\af_{x_\mu}=\af$ and $f(x_\mu)=\mu$.

To show that $f$ is continuous, observe that  if $\af \in \CL$ and $A \in \CR_\af$, then $A \subseteq V_{\af^{-1}}$ and $A'=\rho_\af(A)$ is such that $f(A')=\CN(\af,A)$.  More generally, for $\af\in\CW^{\geq 1}$ and $A \in \CR_\af$, we have  $A \subseteq V_{\af_n^{-1}}$, and one then also can see that   $A'=\rho_\af(A)$ is such that $f(A')=\CN(\af,A)$. Lastly,  for  $A \in \CB$, we have 
\begin{align}\label{open to open}f(A)=\CN(\emptyset, A).
\end{align}
Then, for $\af \in \CW^*   $, $A \in \CR_\af$, $\af_1, \cdots, \af_n \in \CW^* $ and $A_i \in \CR_{\af_i}$ for $i=1, \cdots, n$, we have 
$$f^{-1}\Big(\CN(\af,A)\setminus(\cup_{i=1}^n\CN(\af_i,A_i))\Big)=A'\setminus(\cup_{i=1}^n A'_i)\in\CB,$$
where $A'=\rho_\af(A)$ and $A_i'=\rho_{\af_i}(A_i)$ for $i=1, \cdots, n$. 
Now, by the observation  in the proof of Lemma \ref{compact open set} and Remark \ref{basis of T}, we have a basis of $\partial E$ consisting of sets of the form $\CN(\af,A)\setminus(\cup_{i=1}^n\CN(\af_i,A_i))$. Thus, $f$ is continuous. 
 Also, By (\ref{open to open}), we see that $f$ is an open map. Hence, $f$ is a homeomorphism.

 To prove that $f$ is equivariant, it is enough to prove that 
\begin{align}\label{equivariant}f(\rho_{\af^{-1}}(x))=\varphi_{\af^{-1}}(f(x))
\end{align} for all $\af \in \CL$ and all $x \in V_\af$ since $\rho$ and $\varphi$ are semi-saturated. To show (\ref{equivariant}), we only need to check that  $r(\varphi_{\af^{-1}}(f(x)))$ is the ultrafilter in $\CB$ corresponding to $\rho_{\af^{-1}}(x)$.   Now, if $f(x)=(e^{\af_i}_{\eta_i})_{1 \leq i \leq |\af_x|}$, where $\af_1=\af$ and $\eta_1=\{A \in \CR_\af: \rho_{\af^{-1}}(x) \in A\}$, then 
we have $\varphi_{\af^{-1}}(f(x)) =(e^{\af_i}_{\eta_i})_{2 \leq i \leq |\af_x|}$. It thus easily follows that 
\begin{align*} r(\varphi_{\af^{-1}}(f(x))) & =h_{[\af]\emptyset}(\eta_1) \\
& =\{A \in \CB: \rho_{\af^{-1}}(x)\in  A \}.
\end{align*}

Lastly, since $f$ is an equivariant homeomorphism, it is straightforward to check that the map $F:  \mathbb{F} \ltimes_\rho X \to  \mathbb{F} \ltimes_\varphi \partial E $  given by $F((x,t,y))= (f(x),t,f(y))$ is an isomorphism of topological groupoids. 
\end{proof}

We now look at the C*-algebraic version of Theorem \ref{partial action model}. For that, we prove a few lemmas first.

\begin{lem}\label{orthogonal projections}
Let $\CA$ be a commutative C*-algebra and $\{p_1,\cdots,p_n\}$ a family of projections. Then, there exists a family of mutually orthogonal projections $\{q_1,\cdots,q_m\}$ such that for all $i=1,\ldots,n$, there is $I_i\subseteq \{1,\ldots,m\}$ such that $p_i=\sum_{j\in I_i}q_j$.
\end{lem}

\begin{proof}
This follows a usual disjointification of sets. See for instance \cite[Lemma 4.1]{BC1}.
\end{proof}

\begin{lem}\label{commutative projections}
Let $\CA$ be a commutative C*-algebra generated by its projections and $\CB$ the set of all projections of $\CA$. Then $\CB$ is a Boolean algebra with operations given by $p\cap q=pq$, $p\cup q=p+q-pq$ and $p\setminus q=p-pq$ for $p,q\in\CB$. Moreover $\CA\cong C_0(\WCB)$, where $\WCB$ is the Stone dual of $\CB$.
\end{lem}

\begin{proof}
Let $\CA_0$ be the *-subalgebra generated by the projections of $\CA$. By Lemma \ref{orthogonal projections}, an non-zero element of $a\in\CA_0$ can be written as linear combination of non-zero orthogonal projections. Suppose that $a=\sum_{i=1}^n\lambda_ip_i\in\CA_0$, where $p_ip_j=0$ if $i\neq j$ and $\lambda_i\neq 0$ for all $i$, and let $f_a:\WCB\to\mathbb{C}$ be the function given by $f_a=\sum_{i=1}^n\lambda_i1_{Z(p_i)}$. Note that
\begin{equation}\label{norm a f_a}\|a\|=\max_{i=1,\ldots,n}|\lambda_i|=\|f_a\|\end{equation}
since $Z(p_i)\cap Z(p_j)=\emptyset$ if $i\neq j$. We claim that $f_a$ does not depend on the choice of decomposition for $a$. Suppose then that $a=\sum_{j=1}^m\sigma_j q_j$, where $q_iq_j=0$ if $i\neq j$, and $\sigma_j\neq 0$ for all $j$. We apply Lemma \ref{orthogonal projections} in the family $\{p_1,\ldots,p_n,q_1,\ldots,q_m\}$ to find a new family of non-zero mutually orthogonal projections $\{r_1,\ldots,r_t\}$, which is such that for each $i=1,\ldots,n$, there is $I_i\subseteq \{1,\ldots,t\}$ such that $p_i=\sum_{k\in I_i}r_k$, and for each $j=1,\ldots,m$, there is $J_j\subseteq \{1,\ldots,t\}$ such that $q_j=\sum_{k\in J_j}r_k$. Note that if $i\neq i'$, then $I_i\cap I_{i'}=\emptyset$ because $p_ip_{i'}=0$. We can also assume that $\bigcup_{i=1}^n I_i=\{1,\ldots,t\}$, for if there is $k$ such that $k\notin I_i$ for all $i$, then $r_kp_i=0$ for all $i$, and hence $ar_k=0$. This would imply that $r_kq_j=0$ for all $j$ and we can simply remove $r_k$ from the list. Similarly $J_j\cap J_{j'}=\emptyset$ if $j\neq j'$ and $\bigcup_{j=1}^m J_j=\{1,\ldots,t\}$. Moreover if $I_i\cap J_j\neq \emptyset$, then $\lambda_i=\sigma_j$. For each $k=1,\ldots,t$, if we let $\zeta_k:=\lambda_i=\sigma_j$ for the unique $i,j$ such that $k\in I_i\cap J_j$, we see that
\[\sum_{i=1}^n\lambda_i 1_{Z(p_i)}=\sum_{k=1}^t\zeta_k 1_{Z(r_k)}=\sum_{j=1}^m\sigma_j 1_{Z(q_j)},\]
proving the claim. 

The map $\Phi_0:\CA_0\to C_0(\WCB)$ given by $\Phi_0(a)=f_a$ is then well-defined. It is easy to see that $\Phi_0$ is linear and preserves involution. Equation \eqref{norm a f_a} shows that $\Phi_0$ is an isometry, and an argument using Lemma \ref{orthogonal projections} as above shows that $\Phi_0$ is multiplicative. Since the set $\{1_{Z(p)}:p\in\CB\}$ generates $C_0(\WCB)$, $\Phi_0$ extends to a *-isomorphism from $\CA$ onto $C_0(\WCB)$.
\end{proof}

\begin{lem}\label{complement ideal}
Suppose that $\CA=C_0(X)$ for some locally compact Hausdorff space $X$ and let $U$ be an open subset of $X$. Then, $U$ is closed if and only if there is an ideal $J$ of $\CA$ such that $\CA=C_0(U)\oplus J$.
\end{lem}

\begin{proof}
Suppose first that $U$ is closed and let $V=X\setminus U$. Then $V$ is clopen and clearly $C_0(U)\cap C_0(V)=\{0\}$. Given $f\in C_0(X)$, if we define $f_U:X\to\mathbb{C}$ by $f_U(x)=f(x)$ if $x\in U$ and $f_U(x)=0$ if $x\notin U$, then, because $U$ is clopen, $f_U\in C_0(U)$. Similarly we define $f_V$ so that $f=f_U+f_V$. It follows that $\CA=C_0(U)\oplus C_0(V)$.

Suppose now that $\CA=C_0(U)\oplus J$ for some ideal $J$. It is well-know that $J=C_0(V)$ for some open subset $V$ of $X$. Note that if $U\cap V\neq \emptyset$, then there is $f\in C_0(U)\cap C_0(V)$, and if $X\setminus (U\cup V)\neq\emptyset$, then there is $f\in C_0(X)\setminus(C_0(U)\oplus C_0(V))$, both being contradictions. Hence $U=X \setminus V$ and $U$ is closed.
\end{proof}

\begin{cor}\label{partial model C*-algebra}
Let $\CA$ be a commutative C*-algebra generated by its projections and $\rho=(\{D_t\}_{t \in \mathbb{F}}, \{\tau_t\}_{t \in \mathbb{F}})$ be a semi-saturated, orthogonal topological partial action of a free group $\mathbb{F}$ on $\CA$ such that for every generator $\af$ of $\mathbb{F}$, we have $\CA=D_\af\oplus C_\af$ for some ideal $C_\af$ of $\CA$. Then, there exists a Boolean dynamical system $(\CB,\CL,\theta)$ such that $\CA\ltimes_{\tau}\mathbb{F}\cong C^*(\CB,\CL,\theta)$.
\end{cor}

\begin{proof}
As discussed above, $\CA\cong C_0(X)$ for some Stone space. The result then follows from Lemmas \ref{commutative projections} and \ref{complement ideal}, the Gelfand duality, Corollary \ref{cor: iso} and Theorem \ref{partial action model}. 
\end{proof}

\section{Graded and gauge-invariant ideals}\label{sec: ideals}

As an application of the results of Section \ref{section model}, we prove that if $(\CB,\CL,\theta,\CI_\af)$ is a generalized Boolean dynamical system and $I$ is a gauge-invariant ideal of $C^*(\CB,\CL,\theta,\CI_\af)$, then $I$ itself is a C*-algebra of a generalized Boolean dynamical system. In order to prove this result, we first prove that gauge-invariant ideals are graded ideals for some suitable gradings.

Let $B$ be a $C^*$-algebra and let $G$ be a discrete group. We say that $B$ is {\it G-graded} if there is a linearly independent family $\{B_g\}_{g \in G}$ of closed linear subspaces of $B$ such that, for all $g,h \in G$, 
\begin{enumerate}
\item $B_gB_h \subseteq B_{gh}$,
\item $B_g^*=B_{g^{-1}},$ and 
\item $\bigoplus_{g \in G} B_g $ is dense in $B$.
\end{enumerate}

Let $B$ be a $G$-graded $C^*$-algebra with grading subspaces $\{B_g\}_{g \in G}$. An ideal $J$ of $B$ is said to be {\it G-graded ideal} if $J=\overline{\bigoplus_{g \in G} J_g}$ where each $J_g=J \cap B_g$.

 By Corollary \ref{cor: iso}, there is a $\mathbb{F}$-grading on $C^*(\CB,\CL,\theta,\CI_\af)$, since $C_0(\partial E) \rtimes_{\hat{\varphi}} \mathbb{F}$ has a natural $\mathbb{F}$-grading given by the partial action. Namely, for $t\in \mathbb{F}$, we set $\big(C_0(\partial E) \rtimes_{\hat{\varphi}} \mathbb{F}\big)_t=C_0(U_t)\delta_t$. There is also a $\mathbb{Z}$-grading that comes from the isomorphism of \cite[Corollary 7.11]{CasK1}. More specifically, this grading is given by $C^*(\CB,\CL,\theta,\CI_\af)_n=\overline{{\rm \operatorname{span}}}\{S_{\af,A}S_{\bt,A}^*: \af,\bt
\in \CL^*,\ |\af|-|\bt|=n ~\text{and}~ A \in \CI_\af\cap \CI_\bt \}$, where $n\in\mathbb{Z}$. Because of Theorem \ref{iso groupoids}, we see that $C^*(\CB,\CL,\theta,\CI_\af)_n$ is the closure of $\bigoplus_{t\in\mathbb{F}, |t|=n}C^*(\CB,\CL,\theta,\CI_\af)_t$.

\begin{prop}\label{graded ideal}
Let $J$ be an ideal of $C^*(\CB,\CL,\theta,\CI_\af)$. The following are equivalent:
\begin{enumerate}
    \item $J$ is a $\mathbb{F}$-graded ideal.
    \item $J$ is a $\mathbb{Z}$-graded ideal.
    \item $J$ is a gauge-invariant ideal.
\end{enumerate}
\end{prop}

\begin{proof}
Note that the $\emptyset$-component with respect to the $\mathbb{F}$-grading is contained in the $0$-component with respect to the $\mathbb{Z}$-grading so that (1) implies (2). 

Suppose that $J$ is a $\mathbb{Z}$-graded ideal. For each $n\in\mathbb{Z}$, $z\in\mathbb{T}$ and $x\in J\cap C^*(\CB,\CL,\theta,\CI_\af)_n$, we have that $\gamma_z(x)=z^n x\in J\cap C^*(\CB,\CL,\theta,\CI_\af)_n$. It follows that $$\gamma_z\left(\bigoplus_{n\in\mathbb{Z}} J\cap C^*(\CB,\CL,\theta,\CI_\af)_n\right)\subseteq \bigoplus_{n\in\mathbb{Z}} J\cap C^*(\CB,\CL,\theta,\CI_\af)_n.$$
Since $J$ is $\mathbb{Z}$-graded and $\gamma_z$ is continuous, we conclude that $\gamma_z(J)\subseteq J$ and hence $J$ is gauge-invariant.

Finally, if $J$ is gauge-invariant. By \cite[Proposition 7.3]{CaK2}, $J$ is generated by elements of the $\emptyset$-component of $J$ and hence $J$ is $\mathbb{F}$-graded.
\end{proof}

\begin{cor}\label{cor: gauge-invariant ideal}
Let $(\CB,\CL,\theta,\CI_\af)$ be a generalized Boolean dynamical system and $J$ be a gauge-invariant ideal of $C^*(\CB,\CL,\theta,\CI_\af)$. Then, there exists a generalized Boolean dynamical system $(\CB^J,\CL^J,\theta^J,\CI_\af^J)$ such that $J\cong C^*(\CB^J,\CL^J,\theta^J,\CI_\af^J)$.
\end{cor}

\begin{proof}
By Proposition \ref{graded ideal}, $J$ is $\mathbb{F}$-graded. By Corollary \ref{cor: iso}, we can see $C_0(\partial E)$ as a subalgebra of $C^*(\CB,\CL,\theta,\CI_\af)$. Then, for $I=J\cap C_0(\partial E)$ we obtain an ideal of $C_0(\partial E)$ which is $\hat{\varphi}$-invariant by \cite[Proposition~23.11]{ExelBook} and such that $J=\left< I \right>$ by \cite[Proposition~23.1]{ExelBook}. By \cite[Proposition~3.1]{ELQ2002}, we have that $J\cong I\rtimes_{\hat{\varphi}}\mathbb{F}$. Also, there is an open set $U$ of $\partial E$ such that $I\cong C_0(U)$, which is $\varphi$-invariant because $I$ is $\hat{\varphi}$-invariant. Moreover, by \cite{A}, $I\rtimes_{\hat{\varphi}}\mathbb{F}\cong C^*(\mathbb{F}\ltimes_{\varphi} U)$. We note that $U$ is a Stone space because it is a open subset of the Stone space $\partial E$. By Lemmas \ref{open sets} and \ref{U af closed}, the sets $U_\af\cap U$ are clopen with respect to $U$ for all $\af\in\CL$. By Theorem \ref{partial action model}, we obtain a generalized Boolean dynamical system $(\CB^J,\CL^J,\theta^J,\CI_\af^J)$ modelling the partial action of $\mathbb{F}$ on $U$. By Corollary \ref{cor: iso}, we obtain $J\cong C^*(\CB^J,\CL^J,\theta^J,\CI_\af^J)$.
\end{proof}


\begin{thebibliography}{10}
\bibitem{A}
F.~Abadie.
\newblock On partial actions and groupoids.
\newblock {\em Proc. Amer. Math. Soc.}, 132(4):1037--1047, 2004.

\bibitem{BP1}
T.~Bates and D.~Pask.
\newblock {$C\sp *$}-algebras of labelled graphs.
\newblock {\em J. Operator Theory}, 57(1):207--226, 2007.

\bibitem{BC1}
G.~Boava and G.~G. de~Castro.
\newblock Infinite sum relations on universal {$\rm C^*$}-algebras.
\newblock {\em J. Math. Anal. Appl.}, 507(2):Paper No. 125850, 16, 2022.

\bibitem{BCM1}
G.~Boava, G.~G. de~Castro, and F.~de~L.~Mortari.
\newblock Inverse semigroups associated with labelled spaces and their tight
  spectra.
\newblock {\em Semigroup Forum}, 94(3):582--609, 2017.

\bibitem{BCGW}
G.~Boava, G.~G. de~Castro, D.~Gonçalves, and D.~W. van Wyk.
\newblock Leavitt path algebras of labelled graphs.
\newblock {\em Arxiv}, arXiv:2106.06036 [math.RA], 2021.

\bibitem{CaK2}
T.~M. Carlsen and E.~J. Kang.
\newblock Gauge-invariant ideals of {$C^*$}-algebras of {B}oolean dynamical
  systems.
\newblock {\em J. Math. Anal. Appl.}, 488(1):124037, 27, 2020.

\bibitem{COP}
T.~M. Carlsen, E.~Ortega, and E.~Pardo.
\newblock {$C^*$}-algebras associated to {B}oolean dynamical systems.
\newblock {\em J. Math. Anal. Appl.}, 450(1):727--768, 2017.

\bibitem{CK}
J.~Cuntz and W.~Krieger.
\newblock A class of {$C^{\ast} $}-algebras and topological {M}arkov chains.
\newblock {\em Invent. Math.}, 56(3):251--268, 1980.

\bibitem{CasK1}
G.~G. de~Castro and E.~J. Kang.
\newblock Boundary path groupoids of generalized boolean dynamical systems and
  their C*-algebras.
\newblock {\em Arxiv}, arXiv:2106.09366 [math.OA], 2021.

\bibitem{Exel1}
R.~Exel.
\newblock Inverse semigroups and combinatorial {$C\sp \ast$}-algebras.
\newblock {\em Bull. Braz. Math. Soc. (N.S.)}, 39(2):191--313, 2008.

\bibitem{ExelBook}
R.~Exel.
\newblock {\em Partial dynamical systems, {F}ell bundles and applications},
  volume 224 of {\em Mathematical Surveys and Monographs}.
\newblock American Mathematical Society, Providence, RI, 2017.

\bibitem{ExelLaca2003}
R.~Exel and M.~Laca.
\newblock Partial dynamical systems and the {KMS} condition.
\newblock {\em Comm. Math. Phys.}, 232(2):223--277, 2003.

\bibitem{ELQ2002}
R.~Exel, M.~Laca, and J.~Quigg.
\newblock Partial dynamical systems and {$C^*$}-algebras generated by partial
  isometries.
\newblock {\em J. Operator Theory}, 47(1):169--186, 2002.

\bibitem{Ka2004}
T.~Katsura.
\newblock A class of {$C^\ast$}-algebras generalizing both graph algebras and
  homeomorphism {$C^\ast$}-algebras. {I}. {F}undamental results.
\newblock {\em Trans. Amer. Math. Soc.}, 356(11):4287--4322, 2004.

\bibitem{Ka2021}
T.~Katsura.
\newblock Topological graphs and singly generated dynamical systemss.
\newblock {\em Arxiv}, arXiv:2107.01389 [math.OA], 2021.

\bibitem{KL2017}
A.~Kumjian and H.~Li.
\newblock Twisted topological graph algebras are twisted groupoid
  {$C^*$}-algebras.
\newblock {\em J. Operator Theory}, 78(1):201--225, 2017.

\bibitem{Lawson2012}
M.~V. Lawson.
\newblock Non-commutative {S}tone duality: inverse semigroups, topological
  groupoids and {$C^\ast$}-algebras.
\newblock {\em Internat. J. Algebra Comput.}, 22(6):1250058, 47, 2012.

\bibitem{Ren}
J.~Renault.
\newblock {\em A groupoid approach to {$C^{\ast} $}-algebras}, volume 793 of
  {\em Lecture Notes in Mathematics}.
\newblock Springer, Berlin, 1980.




\end{thebibliography}
\end{document}